\documentclass{amsart}
\usepackage{amsfonts}
\usepackage{amsmath,amscd}
\usepackage{amsthm}
\usepackage{amssymb}
\usepackage{latexsym}
\usepackage{mathrsfs}
\usepackage[usenames,dvipsnames]{color}

\allowdisplaybreaks

\newtheorem{thm}{Theorem}[section]
\newtheorem{cor}[thm]{Corollary}
\newtheorem{lem}[thm]{Lemma}

\newtheorem{prop}[thm]{Proposition}

\newtheorem{rem}[thm]{Remark}

\usepackage[normalem]{ulem}

\numberwithin{equation}{section}
\begin{document}
\newcommand{\beqa}{\begin{eqnarray}}
\newcommand{\eeqa}{\end{eqnarray}}
\newcommand{\thmref}[1]{Theorem~\ref{#1}}
\newcommand{\secref}[1]{Sect.~\ref{#1}}
\newcommand{\lemref}[1]{Lemma~\ref{#1}}
\newcommand{\propref}[1]{Proposition~\ref{#1}}
\newcommand{\corref}[1]{Corollary~\ref{#1}}
\newcommand{\remref}[1]{Remark~\ref{#1}}
\newcommand{\er}[1]{(\ref{#1})}
\newcommand{\nc}{\newcommand}
\newcommand{\rnc}{\renewcommand}

\nc{\cal}{\mathcal}

\nc{\goth}{\mathfrak}
\rnc{\bold}{\mathbf}
\renewcommand{\frak}{\mathfrak}
\renewcommand{\Bbb}{\mathbb}

\newcommand{\id}{\mathrm{id}}
\nc{\Cal}{\mathcal}
\nc{\Xp}[1]{X^+(#1)}
\nc{\Xm}[1]{X^-(#1)}
\nc{\on}{\operatorname}
\nc{\ch}{\mbox{ch}}
\nc{\Z}{{\Bbb Z}}
\nc{\J}{{\mathcal J}}
\nc{\Q}{{\Bbb Q}}
\renewcommand{\P}{{\mathcal P}}
\nc{\K}{{\Bbb K}}
\nc{\N}{{\Bbb N}}
\nc\beq{\begin{equation}}
\nc\enq{\end{equation}}
\nc\lan{\langle}
\nc\ran{\rangle}
\nc\bsl{\backslash}
\nc\mto{\mapsto}
\nc\lra{\leftrightarrow}
\nc\hra{\hookrightarrow}
\nc\sm{\smallmatrix}
\nc\esm{\endsmallmatrix}
\nc\sub{\subset}
\nc\ti{\tilde}
\nc\nl{\newline}
\nc\fra{\frac}
\nc\und{\underline}
\nc\ov{\overline}
\nc\ot{\otimes}
\nc\qOns{\mathcal{B}_c}
\nc\bbq{\bar{\bq}_l}
\nc\bcc{\thickfracwithdelims[]\thickness0}
\nc\ad{\text{\rm ad}}
\nc\Ad{\text{\rm Ad}}
\nc\Hom{\text{\rm Hom}}
\nc\End{\text{\rm End}}
\nc\Ind{\text{\rm Ind}}
\nc\Res{\text{\rm Res}}
\nc\Ker{\text{\rm Ker}}
\rnc\Im{\text{Im}}
\nc\sgn{\text{\rm sgn}}
\nc\tr{\text{\rm tr}}
\nc\Tr{\text{\rm Tr}}
\nc\supp{\text{\rm supp}}
\nc\card{\text{\rm card}}
\nc\bst{{}^\bigstar\!}
\nc\he{\heartsuit}
\nc\clu{\clubsuit}
\nc\spa{\spadesuit}
\nc\di{\diamond}
\nc\cW{\cal W}
\nc\cG{\cal G}
\nc\al{\alpha}
\nc\bet{\beta}
\nc\ga{\gamma}
\nc\de{\delta}
\nc\ep{\epsilon}
\nc\io{\iota}
\nc\om{\omega}
\nc\si{\sigma}
\rnc\th{\theta}
\nc\ka{\kappa}
\nc\la{\lambda}
\nc\ze{\zeta}

\nc\vp{\varpi}
\nc\vt{\vartheta}
\nc\vr{\varrho}

\nc\Ga{\Gamma}
\nc\De{\Delta}
\nc\Om{\Omega}
\nc\Si{\Sigma}
\nc\Th{\Theta}
\nc\La{\Lambda}

\nc\boa{\bold a}
\nc\bob{\bold b}
\nc\boc{\bold c}
\nc\bod{\bold d}
\nc\boe{\bold e}
\nc\bof{\bold f}
\nc\bog{\bold g}
\nc\boh{\bold h}
\nc\boi{\bold i}
\nc\boj{\bold j}
\nc\bok{\bold k}
\nc\bol{\bold l}
\nc\bom{\bold m}
\nc\bon{\bold n}
\nc\boo{\bold o}
\nc\bop{\bold p}
\nc\boq{\bold q}
\nc\bor{\bold r}
\nc\bos{\bold s}
\nc\bou{\bold u}
\nc\bov{\bold v}
\nc\bow{\bold w}
\nc\boz{\bold z}

\nc\ba{\bold A}
\nc\bb{\bold B}
\nc\bc{\bold C}
\nc\bd{\bold D}
\nc\be{\bold E}
\nc\bg{\bold G}
\nc\bh{\bold H}
\nc\bi{\bold I}
\nc\bj{\bold J}
\nc\bk{\bold K}
\nc\bl{\bold L}
\nc\bm{\bold M}
\nc\bn{\bold N}
\nc\bo{\bold O}
\nc\bp{\bold P}
\nc\bq{\bold Q}
\nc\br{\bold R}
\nc\bs{\bold S}
\nc\bt{\bold T}
\nc\bu{\bold U}
\nc\bv{\bold V}
\nc\bw{\bold W}
\nc\bz{\bold Z}
\nc\bx{\bold X}

\nc\ca{\mathcal A}
\nc\cb{\mathcal B}
\nc\cc{\mathcal C}
\nc\cd{\mathcal D}
\nc\ce{\mathcal E}
\nc\cf{\mathcal F}
\nc\cg{\mathcal G}
\rnc\ch{\mathcal H}
\nc\ci{\mathcal I}
\nc\cj{\mathcal J}
\nc\ck{\mathcal K}
\nc\cl{\mathcal L}
\nc\cm{\mathcal M}
\nc\cn{\mathcal N}
\nc\co{\mathcal O}
\nc\cp{\mathcal P}
\nc\cq{\mathcal Q}
\nc\car{\mathcal R}
\nc\cs{\mathcal S}
\nc\ct{\mathcal T}
\nc\cu{\mathcal U}
\nc\cv{\mathcal V}
\nc\cz{\mathcal Z}
\nc\cx{\mathcal X}
\nc\cy{\mathcal Y}

\nc\e[1]{E_{#1}}
\nc\ei[1]{E_{\delta - \alpha_{#1}}}
\nc\esi[1]{E_{s \delta - \alpha_{#1}}}
\nc\eri[1]{E_{r \delta - \alpha_{#1}}}
\nc\ed[2][]{E_{#1 \delta,#2}}
\nc\ekd[1]{E_{k \delta,#1}}
\nc\emd[1]{E_{m \delta,#1}}
\nc\erd[1]{E_{r \delta,#1}}

\nc\ef[1]{F_{#1}}
\nc\efi[1]{F_{\delta - \alpha_{#1}}}
\nc\efsi[1]{F_{s \delta - \alpha_{#1}}}
\nc\efri[1]{F_{r \delta - \alpha_{#1}}}
\nc\efd[2][]{F_{#1 \delta,#2}}
\nc\efkd[1]{F_{k \delta,#1}}
\nc\efmd[1]{F_{m \delta,#1}}
\nc\efrd[1]{F_{r \delta,#1}}

\nc\fa{\frak a}
\nc\fb{\frak b}
\nc\fc{\frak c}
\nc\fd{\frak d}
\nc\fe{\frak e}
\nc\ff{\frak f}
\nc\fg{\frak g}
\nc\fh{\frak h}
\nc\fj{\frak j}
\nc\fk{\frak k}
\nc\fl{\frak l}
\nc\fm{\frak m}
\nc\fn{\frak n}
\nc\fo{\frak o}
\nc\fp{\frak p}
\nc\fq{\frak q}
\nc\fr{\frak r}
\nc\fs{\frak s}
\nc\ft{\frak t}
\nc\fu{\frak u}
\nc\fv{\frak v}
\nc\fz{\frak z}
\nc\fx{\frak x}
\nc\fy{\frak y}

\nc\fA{\frak A}
\nc\fB{\frak B}
\nc\fC{\frak C}
\nc\fD{\frak D}
\nc\fE{\frak E}
\nc\fF{\frak F}
\nc\fG{\frak G}
\nc\fH{\frak H}
\nc\fJ{\frak J}
\nc\fK{\frak K}
\nc\fL{\frak L}
\nc\fM{\frak M}
\nc\fN{\frak N}
\nc\fO{\frak O}
\nc\fP{\frak P}
\nc\fQ{\frak Q}
\nc\fR{\frak R}
\nc\fS{\frak S}
\nc\fT{\frak T}
\nc\fU{\frak U}
\nc\fV{\frak V}
\nc\fZ{\frak Z}
\nc\fX{\frak X}
\nc\fY{\frak Y}
\nc\tfi{\ti{\Phi}}
\nc\bF{\bold F}
\rnc\bol{\bold 1}

\nc\ua{\bold U_\A}

\nc\qinti[1]{[#1]_i}
\nc\q[1]{[#1]_q}
\nc\xpm[2]{E_{#2 \delta \pm \alpha_#1}}  
\nc\xmp[2]{E_{#2 \delta \mp \alpha_#1}}
\nc\xp[2]{E_{#2 \delta + \alpha_{#1}}}
\nc\xm[2]{E_{#2 \delta - \alpha_{#1}}}
\nc\hik{\ed{k}{i}}
\nc\hjl{\ed{l}{j}}
\nc\qcoeff[3]{\left[ \begin{smallmatrix} {#1}& \\ {#2}& \end{smallmatrix}
\negthickspace \right]_{#3}}
\nc\qi{q}
\nc\qj{q}

\nc\ufdm{{_\ca\bu}_{\rm fd}^{\le 0}}


\nc\isom{\cong} 

\nc{\pone}{{\Bbb C}{\Bbb P}^1}
\nc{\pa}{\partial}
\def\H{\mathcal H}
\def\L{\mathcal L}
\nc{\F}{{\mathcal F}}
\nc{\Sym}{{\goth S}}
\nc{\A}{{\mathcal A}}
\nc{\arr}{\rightarrow}
\nc{\larr}{\longrightarrow}

\nc{\ri}{\rangle}
\nc{\lef}{\langle}
\nc{\W}{{\mathcal W}}
\nc{\uqatwoatone}{{U_{q,1}}(\su)}
\nc{\uqtwo}{U_q(\goth{sl}_2)}
\nc{\dij}{\delta_{ij}}
\nc{\divei}{E_{\alpha_i}^{(n)}}
\nc{\divfi}{F_{\alpha_i}^{(n)}}
\nc{\Lzero}{\Lambda_0}
\nc{\Lone}{\Lambda_1}
\nc{\ve}{\varepsilon}
\nc{\phioneminusi}{\Phi^{(1-i,i)}}
\nc{\phioneminusistar}{\Phi^{* (1-i,i)}}
\nc{\phii}{\Phi^{(i,1-i)}}
\nc{\Li}{\Lambda_i}
\nc{\Loneminusi}{\Lambda_{1-i}}
\nc{\vtimesz}{v_\ve \otimes z^m}

\nc{\asltwo}{\widehat{\goth{sl}_2}}
\nc\ag{\widehat{\goth{g}}}  
\nc\teb{\tilde E_\boc}
\nc\tebp{\tilde E_{\boc'}}

\newcommand{\LR}{\bar{R}}
\newcommand{\eeq}{\end{equation}}
\newcommand{\ben}{\begin{eqnarray}}
\newcommand{\een}{\end{eqnarray}}

\newcommand{\C}{\mathbb{C}}
\newcommand{\R}{\mathbb{R}}
\newcommand{\Rbar}{\overline{\roots}}
\newcommand{\gfrak}{{\mathfrak{g}}}
\newcommand{\ogamma}{\overline{\gamma}}
\newcommand{\ons}{O}
\newcommand{\ima}{{\mathrm{im}}}
\newcommand{\real}{{\mathrm{re}}}
\newcommand{\uqg}{U_q(\gfrak)}
\newcommand{\roots}{{\mathcal R}}
\newcommand{\slz}{{\mathfrak{sl}}_2(\C)}
\newcommand{\slzh}{\widehat{\mathfrak{sl}}_2}
\newcommand{\thetah}{\widehat{\theta}}
\newcommand{\uqslzh}{U_q(\slzh)}

\newcommand \red {\textcolor{red}}

\title[Root vectors for the $\lowercase{q}$-Onsager algebra]{Braid group action and root vectors for the $\lowercase{q}$-Onsager algebra}
\thanks{Pascal Baseilhac is supported by C.N.R.S. Stefan Kolb was supported by EPSRC grant EP/K025384/1.}

\author{Pascal Baseilhac}
\address{Institut Denis-Poisson CNRS/UMR 7013 - Universit\'e de Tours - Universit\'e d'Orl\'eans
Parc de Grammont, 37200 Tours, FRANCE}
\email{pascal.baseilhac@idpoisson.fr}

\author{Stefan Kolb}
\address{School of Mathematics, Statistics and Physics, Newcastle University, Newcastle upon Tyne NE1 7RU, UK}
\email{stefan.kolb@newcastle.ac.uk}

\subjclass[2010]{17B37; 81R50}
\keywords{Quantum groups, $q$-Onsager algebra, braid group}

\begin{abstract}
  We define two algebra automorphisms $T_0$ and $T_1$ of the $q$-Onsager algebra $\cb_c$, which provide an analog of G.~Lusztig's braid group action for quantum groups. These automorphisms are used to define root vectors which give rise to a PBW basis for $\cb_c$. We show that the root vectors satisfy $q$-analogs of Onsager's original commutation relations. The paper is much inspired by I.~Damiani's construction and investigation of root vectors for the quantized enveloping algebra of $\slzh$.
\end{abstract}

\maketitle
\section{Introduction}
The Onsager algebra $\ons$ appeared first in 1944 in L.~Onsager's investigation of the two-dimensional Ising model \cite{a-Onsager44}. It is an infinite dimensional Lie algebra with two natural presentations in terms of generators and relations. Onsager's original definition provides a linear basis $\{A_n, G_m\,|\,n\in \Z, m\in \N\}$ and the commutators
\begin{align}
    [A_n,A_m]= 4 G_{n-m}, \quad [G_n,G_m]=0, \quad
    [G_m,A_n] = 2 A_{n+m} - 2 A_{n-m}. \label{eq:Onsager}
\end{align}
L.~Dolan and M.~Grady uncovered a second presentation in terms of only two Lie algebra generators $A_0, A_1$ and the defining relations
\begin{align}\label{eq:DG}
    [A_0,[A_0,[A_0,A_1]]]&= 16 [A_0,A_1], &   [A_1,[A_1,[A_1,A_0]]]&= 16 [A_1,A_0],
\end{align}
see \cite{a-DolanGrady82}. With the advent of the general theory of Kac-Moody algebras \cite{b-Kac1}, it became clear that the Onsager algebra is isomorphic to the Lie subalgebra of the affine Lie algebra $\slzh$ consisting of all elements fixed under the Chevalley involution. The loop realization of $\slzh$ leads to the presentation \eqref{eq:Onsager} while the realization in terms of a Cartan matrix leads to the presentation \eqref{eq:DG}. In particular, the two presentations define isomorphic Lie algebras, see also \cite{a-Davies91,a-Roan91}.

The $q$-Onsager algebra $\cb_c$ is a $q$-analog of the universal enveloping algebra $U(\ons)$. In the present paper it depends on a parameter $c\in \Q(q)$ and $q$ is a formal variable. Initially, the $q$-Onsager algebra was defined in terms of generators and $q$-analogs of the Dolan-Grady relations \eqref{eq:DG}, see \cite{a-Terwilliger03}, \cite{a-Baseilhac05a}. The same relations showed up earlier in the context of polynomial association schemes \cite{a-Terwilliger93}. The $q$-Onsager algebra $\cb_c$ can be realized as a left or right coideal subalgebra of the quantized enveloping algebra $U_q(\slzh)$, see \cite{a-BasBel09}, \cite{a-BasBel12}, \cite{a-Kolb14}. It is the simplest example of a quantum symmetric pair coideal subalgebra of affine type, see \cite[Example 7.6]{a-Kolb14}.

It is an open problem to find quantum group analogs of the generators  $A_n, G_m$ and of Onsager's relations \eqref{eq:Onsager}. Attempts to find a current algebra realization of $\cb_c$ in the spirit of Drinfeld's second realization of $U_q(\slzh)$ were made in \cite{a-BasKoz05}, \cite{a-BasShi10}. This was pursued further in \cite{a-BasBel17p} where it is conjectured that $\cb_c$ is isomorphic to a certain quotient of the current algebra ${\mathcal A}_q$ proposed in \cite{a-BasShi10}. The generators of the quotient of ${\mathcal A}_q$, however, do not specialize to the generators $A_n, G_m$ of $\ons$. Nevertheless, their classical analogs generate a Lie algebra that is isomorphic to the Onsager algebra, see \cite[Theorem 2]{a-BasCr18}. Note also that a third presentation of $O$ in the framework of the non-standard classical Yang-Baxter algebra has been recently identified \cite[Theorem 1]{a-BasBelCr17}.

Let $\alpha_0,\alpha_1$ denote the simple roots for $\slzh$ and set $\delta=\alpha_0+\alpha_1$. In the present paper we construct quantum group analogs $\{B_{n\delta+\alpha_1}, B_{m \delta} \, | \, n\in \Z, m\in \N\}$ of Onsager's generators $\{A_n,G_m\,|\,n\in \Z, m\in\N\}$ of $\ons$. More precisely, let
    \begin{align}\label{eq:slzh-roots}
        \roots=\{n\delta+\alpha_0,n\delta+\alpha_1, m\delta\,|\,n\in \Z, m\in \Z\setminus\{0\}\}
    \end{align}
be the root system of $\slzh$ and set
    \begin{align}\label{eq:Rbar}
       \Rbar=\roots/\sim
    \end{align}
    where $\sim$ is the equivalence relation on  $\roots$ given by $\alpha \sim \beta \Leftrightarrow \alpha=\pm \beta$. We define root vectors $B_\gamma$ for all real roots $\gamma\in \roots^{\real}=\{n\delta\pm\alpha_1\,|\,n\in \Z\}$ and for all positive imaginary roots $\gamma=n\delta$ for $n\in \N$. The root vectors satisfy
    \begin{align}\label{eq:pm-invariance}
      B_\gamma=B_{-\gamma} \qquad \mbox{for all $\gamma\in\roots^{\mathrm{re}}$.}
    \end{align}
For any $\gamma \in \roots$ let $\overline{\gamma}$ denote its image under the canonical projections $\roots\rightarrow \Rbar$. By \eqref{eq:pm-invariance} we may define $B_{\ogamma}=B_\gamma$ for all $\gamma\in \roots^\real\cup\N\delta$ and hence we have a well defined root vector $B_{\ogamma}$ for any $\ogamma\in \Rbar$. Using a filtered-graded argument and results from \cite{a-Damiani93}, we prove a Poincar\'e-Birkhoff-Witt Theorem for $\cb_c$.

\medskip

\noindent {\bf Theorem I.} (Theorem \ref{thm:PBW}) \textit{For any total ordering on $\Rbar$ the ordered monomials in the root vectors $\{B_{\ogamma}\,|\,\ogamma \in \Rbar\}$ form a basis of $\cb_c$.}

\medskip

\noindent
Once an ordering on $\Rbar$ is specified,  we refer to this basis as the \emph{PBW-basis of $\cb_c$}. We show that the root vectors $B_{\ogamma}$ for $\ogamma \in \Rbar$ satisfy commutation relations which are $q$-analogs of the Onsager relations \eqref{eq:Onsager}. To this end note that $(\Z\delta+\alpha_1)\cup \N\delta$ is a representative set of $\Rbar$. We consider the uniquely determined total ordering $<$ on $(\Z\delta+\alpha_1)\cup \N\delta$ and hence on $\Rbar$ defined by
\begin{align}\label{eq:order-def}
  (m+1)\delta+\alpha_1<m\delta+\alpha_1<k\delta<(k+1)\delta
\end{align}  
for all $m\in \Z, k\in\N$. This ordering is adapted to the $q$-Onsager algebra because the corresponding PBW-basis is preserved under an automorphism which appears in the definition of the real root vectors, see Section \ref{sec:choiceR+}. For any $p\in \Q(q)$ and $x,y\in U_q(\slzh)$ define the $p$-commutator by $[x,y]_p=xy-pyx$. Recall the notation $[2]_q=q+q^{-1}$. In the following theorem, which is the main result of this paper, we write the real root vectors as $B_{n\delta+\alpha_1}$ for $n\in \Z$, and we consider the PBW-basis with respect to the total ordering $<$ defined by \eqref{eq:order-def}.

\medskip

\noindent
{\bf Theorem II.} \textit{For any $m,n\in \N$ and $r\in \Z$ the following relations hold
  \begin{align}
   & [B_{m\delta},B_{n\delta}]=0   \label{eq:qOnsCom1}\\
   & [B_{r\delta+\alpha_1}, B_{(r+m)\delta+\alpha_1}]_{q^{-2}}=-B_{m\delta}+ C^{\real}_{r,m}  \label{eq:qOnsCom2}\\
    & [B_{m\delta},B_{r\delta+\alpha_1}] = c [2]_q \big(q^{-2(m-1)} B_{(r-m)\delta+\alpha_1} 
    - q^{2(m-1)}B_{(r+m)\delta+\alpha_1}\big) + C^{\ima}_{r,m} \label{eq:qOnsCom3}
  \end{align}
  where $C^{\real}_{r,m}$ and $C^{\ima}_{r,m}$ are linear combinations of elements of the PBW-basis of $\cb_c$ with coefficients in $(q-1)\big(\Z[q,q^{-1}]+ \Z[q,q^{-1}]c\big)$.}

\medskip

The commutation relations \eqref{eq:qOnsCom1},  \eqref{eq:qOnsCom2}, and \eqref{eq:qOnsCom3} specialize to the Onsager relations \eqref{eq:Onsager} for $q\rightarrow 1$, up to rescaling of the root vectors, see Section \ref{sec:specialization}. The elements $C^{\real}_{r,m}$ and $C^{\ima}_{r,m}$ are given explicitly in Propositions \ref{prop:comRealReal} and \ref{prop:Bb10Bndelbis}.

The construction of the root vectors $B_{\ogamma}$ for $\ogamma\in \Rbar$ and the proof of Theorem II are much inspired by I.~Damiani's construction and investigation of root vectors for $U_q(\slzh)$ in \cite{a-Damiani93}. Damiani constructs root vectors for the positive part $U^+$ of $U_q(\slzh)$. She uses G.~Lusztig's braid group action of the free group in two generators on $U_q(\slzh)$ to construct real root vectors $\{E_{n\delta+\alpha_0}, E_{n\delta+\alpha_1}\,|\,n\in \N_0\}$ in $U^+$. Subsequently she obtains imaginary root vectors $\{E_{n\delta}\,|\,n\in \N\}$ as quadratic expressions in the real root vectors. Let $\roots_+\subset \roots$ denote the set of positive roots. Damiani proves commutation formulas for the $E_\beta$, $\beta\in \roots_+$, by a subtle inductive procedure.

It was conjectured in \cite[Conjecture 1.2]{a-KolPel11} that quantum symmetric pairs of finite type have a natural action of a braid group, and it is reasonable to expect that such an action also exists in the Kac-Moody case. The Onsager algebra is invariant under the action of the braid group of $\slzh$, which is the free group in two generators. Hence we expect to find two suitable algebra automorphisms of the $q$-Onsager algebra $\cb_c$. We construct these automorphisms, $T_0$ and $T_1$, in Section \ref{sec:auto} and use them to define real root vectors $\{B_{n\delta+\alpha_0}, B_{n\delta+\alpha_1}\,|\,n\in \Z\}$ very much in the spirit of \cite{a-Damiani93}. By definition the real root vectors satisfy \eqref{eq:pm-invariance}. The $q$-Onsager algebra $\cb_c$ is filtered with associated graded algebra $U^+$. The imaginary root vectors $\{B_{m\delta}\,|\,m\in\N\}$ are again defined as quadratic expressions in the real root vectors, however these expressions now involve additional terms of lower filter degree, see Section \ref{sec:imroot}. Once all root vectors are defined, the PBW basis of Theorem I is established by a filtered-graded argument using the PBW theorem for $U^+$ in \cite{a-Damiani93} and facts about the structure of quantum symmetric pairs, see \cite{a-Kolb14}. The commutation relations in Theorem II are again proved by an inductive calculation. This calculation is significantly harder than the corresponding calculations in \cite{a-Damiani93} due to the lower order terms in the definition of the imaginary root vectors $B_{m\delta}$.

The fact that the coefficients in Theorem II lie in $(q-1)\big(\Z[q,q^{-1}]+\Z[q,q^{-1}]c\big)$ suggests a connection to integral forms. In particular, it is natural to ask whether the results of the present paper still hold if we relax our assumptions about $q$ to include for example roots of unity. The answer is not straightforward and shall be considered elsewhere. 

The paper is organized as follows. In Section \ref{sec:defs} we recall the definition of the Onsager algebra $\ons$ and the $q$-Onsager algebra $\cb_c$ and we introduce the braid group action on $\cb_c$. The fact that $T_0$ and $T_1$ are indeed well-defined algebra automorphisms of $\cb_c$ is checked by computer calculations which are not reproduced here. In Section \ref{sec:root} we use the automorphisms $T_0$ and $T_1$ to define real and imaginary root vectors in $\cb_c$. We show that, up to a scalar factor, these root vectors specialize to the Onsager generators $A_n$, $G_m$. In Section \ref{sec:PBWBc} we establish Theorem I. We first recall the PBW theorem for $U^+$ as proved in \cite{a-Damiani93}. We then show in Proposition \ref{prop:grad-fil}  that up to a factor the root vectors $B_\gamma\in \cb_c$ for $\gamma\in \roots_+$ project onto Damiani's root vectors $E_\gamma\in U^+$ in the associated graded algebra.

In Section \ref{sec:commutators}, which forms the bulk of the paper, we prove the commutation relations of Theorem II explicitly. We first deal with the $q^{-2}$-commutator of two real root vectors in Proposition \ref{prop:comRealReal}. The hardest part are the commutators of a real with an imaginary root vector which are established in Proposition \ref{prop:Bb10Bndelbis}. These commutators involve a crucial term $F_n$ for $n\ge 2$. The term $F_n$ satisfies a recursive formula, the proof of which is deferred to Appendix \ref{app:Fn}.

\medskip

\noindent{\bf Acknowledgements.} The authors are grateful to Istv\'an Heckenberger for checking the existence of the algebra automorphisms $T_0$ and $T_1$ with the computer algebra program FELIX in July 2013. We are much indebted to Travis Scrimshaw for pointing out an omission in the calculation of commutators of real root vectors in a previous version of this paper.
The authors also wish to thank Marta Mazzocco and Paul Terwilliger for their interest in this work and for comments. Finally, we wish to thank all referees of this paper for their valuable comments. One referee in particular provided a five page report of rare quality which led to a complete revision of the paper. This report contained the suggestion to  express the commutation relations in Theorem II in terms of the PBW basis with respect to the ordering \eqref{eq:order-def} or an equivalent ordering. Moreover, this report contained the subtle Remark \ref{rem:referee} and a much simplified proof of Proposition \ref{prop:Bdel-commute}. We are very grateful for these observations which have streamlined the presentation and shortened the paper significantly. 

PB visited the the School of Mathematics and Statistics at Newcastle University in July 2013 when this project started out. SK visited the Laboratoire de Math\'ematiques et Physique Th\'eorique at Universit\'e Tours in May 2014 and in April 2017. Both authors are grateful to the hosting institutions for the hospitality and the good working conditions. 

\vspace{.2cm}

\noindent{\bf Note.} When we were in the final stages of writing the present paper, Paul Terwilliger published the preprint version of \cite{a-Terwilliger17p} which proves the existence of the algebra automorphisms $T_0$ and $T_1$ without the use of computer calculations.

Travis Scrimshaw has used the results of the present paper to implement the $q$-Onsager algebra in SageMath. His implementation is presently awaiting approval.

\section{Braid group action on the $q$-Onsager algebra}\label{sec:defs}
In this introductory section we recall the definition of the Onsager algebra $\ons$ and its $q$-analog $\cb_c$. We then define the algebra automorphisms $T_0$ and $T_1$ of $\cb_c$, which are analogs of the Lusztig automorphisms for $U_q(\slzh)$.
\subsection{The Onsager algebra}
  Let $e,f,h$ denote the standard generators of the Lie algebra $\slz$. The Chevalley involution $\theta:\slz\rightarrow \slz$ is the involutive Lie algebra automorphism determined by
  \begin{align*}
    \theta(e)&=-f, & \theta(f)&=-e, & \theta(h)=-h.
  \end{align*}
  We are interested in the affine Lie algebra $\slzh= \C[t,t^{-1}]\otimes \slz\oplus \C K \oplus \C d$ with the usual Lie bracket
  \begin{align}
  [t^m\ot x, t^n\ot y]&=t^{m+n}\ot [x,y] + m \delta_{m,-n}(x,y)K, \label{eq:com1Lie}\\
  [d,t^n\ot x]&=n t^n\ot x, \qquad
  K \mbox{ is central, } \nonumber
\end{align}
where $(\cdot, \cdot)$ denotes the symmetric invariant bilinear form
on $\slz$ with $(e,f)=1$, see \cite[Chapter 7]{b-Kac1}. The Chevalley involution $\thetah:\slzh\rightarrow \slzh$ is given by
  \begin{align*}
    \thetah(t^n\otimes x)=t^{-n}\ot \theta(x), \quad \thetah(K)=-K, \quad \thetah(d)=-d.
  \end{align*} 
The Onsager algebra $\ons$ is the infinite dimensional Lie subalgebra of $\slzh$ defined by
\begin{align*}
  \ons=\{a\in \slzh\,|\, \thetah(a)=a\}.
\end{align*}  
The triangular decomposition of $\slzh$ implies that the following elements form a basis of $\ons$ as a complex vector space
\begin{equation}\label{eq:Ons-basis}
\begin{aligned}
  A_n &= 2i (t^{n}\ot e - t^{-n} \ot f)  & & \mbox{for all $n\in \Z$},\\
  G_m &= t^m \ot h - t^{-m} \ot h & & \mbox{for all $m\in \N$}.
\end{aligned}
\end{equation}
Using the relations \eqref{eq:com1Lie} one sees that $A_n$ and $G_m$ satisfy the relations \eqref{eq:Onsager}. 
These relations first appeared in 1944 in Onsager's investigation of the Ising model \cite[(60), (61), (61a)]{a-Onsager44}. Let $\alpha_0, \alpha_1$ denote the simple roots of $\slzh$ and let $\delta=\alpha_0+\alpha_1$ denote the minimal positive imaginary root. Recall that the root system of $\slzh$ is given by \eqref{eq:slzh-roots}. By definition
    \begin{align*}
      t^n\ot e\in (\slzh)_{n\delta+\alpha_1}, \quad
      t^{-n}\ot f\in (\slzh)_{-n\delta-\alpha_1}, \quad
      t^m\ot h\in (\slzh)_{m\delta}
    \end{align*}
and hence the basis \eqref{eq:Ons-basis} of $\ons$ is in one-to-one correspondence to the quotient $\Rbar$ defined by \eqref{eq:Rbar}.     
For all $n\in \Z$ we call $A_n$ the \emph{real root vector} associated to $\overline{n\delta+\alpha_1}\in \Rbar$. Similarly, for $m\in \N$ we call $G_m$ the \emph{imaginary root vector} associated to the $\overline{m\delta}\in \Rbar$.

It follows from \eqref{eq:Onsager} that the Onsager algebra $\ons$ is generated by the elements $A_0$, $A_1$ as a Lie algebra. Defining relations can be seen to be the Dolan-Grady relations \eqref{eq:DG}
which were discovered in \cite{a-DolanGrady82}.
Observe that $A_1$ is the real root vector associated to $\overline{\delta+\alpha_1}$. For our purposes it is more convenient to work with generators which are real root vectors associated to the simple roots $\alpha_0, \alpha_1$. To this end define
\begin{align}\label{eq:D0D1}
  D_0&=-\frac{i}{2} A_{-1}= t^{-1}\otimes e - t\otimes f, &
  D_1&=\frac{i}{2} A_{0}=1\otimes f - 1 \otimes e.
\end{align}
The elements $D_0$ and $D_1$ also generate the Lie algebra $\ons$. The defining relations are now given by
\begin{align}\label{eq:slzhRels}
    [D_0,[D_0,[D_0,D_1]]]&= -4 [D_0,D_1], &   [D_1,[D_1,[D_1,D_0]]]&= -4 [D_1,D_0].
\end{align}
\subsection{The $q$-Onsager algebra $\cb_c$}
Let $\Q(q)$ denote the field of rational functions in an indeterminate $q$ and let $c\in \Q(q)$ such that $c(1)=1$. The $q$-Onsager algebra $\cb_c$ is the unital $\Q(q)$-algebra generated by two elements $B_0$, $B_1$ with the defining relations
\begin{equation}
\begin{aligned}\label{eq:qDG}
    \sum_{i=0}^{3} (-1)^i  \left[ \begin{array}{c} 3 \\  i \end{array}\right]_q   B_0^{3-i} {B_1} B_0^{i} &= -qc(q+q^{-1})^2 (B_0B_1-B_1B_0), \\
\sum_{i=0}^{3} (-1)^i  \left[ \begin{array}{c} 3 \\ i \end{array}\right]_q   {B_1}^{3-i} {B_0} {B_1}^{i} &= -qc(q+q^{-1})^2 (B_1B_0-B_0B_1) 
\end{aligned}
\end{equation}
where we use the usual $q$-binomial coefficients given by 
\begin{align*}
\left[ \begin{array}{c} a \\ b \end{array}\right]_q &=\frac{[a]_q!}{[b]_q!\,[a-b]_q!}\ ,&
  [a]_q!&=[a]_q\,[a-1]_q \dots [1]_q\ ,\\
[b]_q&=\frac{q^b-q^{-b}}{q-q^{-1}},& \quad [0]_q!&=1 \ \qquad \mbox{for $a\in \N$, $b\in \N_0$, $a\ge b$.}\nonumber
\end{align*}
At the specialization $q\rightarrow 1$ the relations \eqref{eq:qDG} transform into the modified Dolan-Grady relations \eqref{eq:slzhRels}.

The $q$-Onsager algebra is the simplest example of a quantum symmetric pair coideal subalgebra for an affine Kac-Moody algebra. Indeed, let $\uqslzh$ denote the Drinfeld-Jimbo quantized enveloping algebra of the affine Kac-Moody algebra $\slzh$ with standard generators $E_0, E_1, F_0, F_1, K_0^{\pm 1}, K_1^{\pm 1}$. Then there exists an algebra embedding
\begin{align*}
  \iota &:\cb_c \rightarrow \uqslzh &&\mbox{with} \qquad \iota(B_i)= F_i - c E_i K_i^{-1} \quad \mbox{for $i\in \{0,1\}$,}
\end{align*}
such that $\iota(\cb_c)$ is a right coideal subalgebra of $\uqslzh$, see \cite[Example 7.6]{a-Kolb14}. To match the conventions in \cite{a-Damiani93} it is preferable to work with the algebra embedding
\begin{align}\label{eq:UqgEmb}
  \iota' &:\cb_c \rightarrow \uqslzh &&\mbox{with} \qquad \iota'(B_i)= E_i - c F_i K_i \quad \mbox{for $i\in \{0,1\}$}
\end{align}
which is given by $\iota'=\omega\circ \iota$ where $\omega\colon U_q(\slzh)\rightarrow U_q(\slzh)$ denotes the algebra automorphism given by $\omega(E_i)=F_i$, $\omega(F_i)=E_i$, $\omega(K_i)=K_i^{-1}$.
Then $\iota'(\cb_c)$ is a left coideal subalgebra of $U_q(\slzh)$, see also \cite[(3.15)]{a-BasBel12}.
\begin{rem}
  In \cite[Example 7.6]{a-Kolb14} the $q$-Onsager algebra depends on 2 parameters $c_0,c_1$. Over a field which contains all square roots these algebras are isomorphic for any parameters. Here we choose to retain one parameter $c=c_0=c_1$ because occasionally different choices for $c$ are convenient. However, we do not keep two parameters $c_0$, $c_1$ because this would complicate the definition of the algebra automorphism $\Phi$ in the upcoming Section \ref{sec:auto}.
\end{rem}
\subsection{Automorphisms of $\cb_c$}\label{sec:auto}
We are interested in three $\Q(q)$-algebra automorphisms of $\cb_c$.
Let $\Phi:\cb_c \rightarrow \cb_c$ denote the $\Q(q)$-algebra automorphism defined by $\Phi(B_0)=B_1$ and $\Phi(B_1)=B_0$. Hence $\Phi$ is obtained from the diagram automorphism of $\uqslzh$ by restriction to $\iota(\cb_c)$. Observe that $\Phi^2=\id$.

The other two automorphisms are obtained from what appears to be a general principle, namely that quantum symmetric pair coideal subalgebras come with an action of a braid group \cite{a-KolPel11}. In the case of $\cb_c$ one expects an action of the braid group of type $A_1^{(1)}$ which is the free group in two generators. This action is closely related to Lusztig's braid group action on $\uqslzh$, \cite[Part VI]{b-Lusztig94}, \cite [Section 2.2]{a-Damiani93}, but it is not merely obtained by restriction to the coideal subalgebra, see Remark \ref{rem:not-mere-restriction} below.
\begin{prop}
  There exists a $\Q(q)$-algebra automorphism $T_0:\cb_c\rightarrow \cb_c$ such that
  \begin{align}
     T_0(B_0)&=B_0,\label{eq:T00}\\ 
     T_0(B_1)&=\frac{1}{q^2[2]_qc}\left(B_1B_0^2 - q[2]_qB_0B_1B_0 + q^2B_0^2B_1\right) 
         + B_1.\label{eq:T01}
  \end{align}
  The inverse automorphism is given by
  \begin{align*}
      T_0^{-1}(B_0)&=B_0,\\
      T_0^{-1}(B_1)&=\frac{1}{q^2[2]_qc}\left(B_0^2B_1 - q[2]_qB_0B_1B_0 + q^2B_1B_0^2\right) 
      + B_1.
  \end{align*}
\end{prop}
This proposition was checked for us by Istv\'an Heckenberger using the computer algebra program FELIX \cite{inp-ApelKlaus}. He verified that $T_0(B_0)$ and $T_0(B_1)$ satisfy the defining relations \eqref{eq:qDG} of $\cb_c$, and similarly for $T_0^{-1}(B_0)$ and $T_0^{-1}(B_1)$, and he confirmed that $T_0(T_0^{-1}(B_1))=B_1=T_0^{-1}(T_0(B_1))$.

The second braid group automorphism $T_1:\cb_c\rightarrow \cb_c$ is defined by
\begin{align}\label{eq:T1PhiT2Phi}
  T_1=\Phi \circ T_0 \circ \Phi.
\end{align}
In words, $T_1$ is obtained from $T_0$ by exchanging subscripts $0$ and $1$ everywhere in \eqref{eq:T00} and \eqref{eq:T01}.

\section{The root vectors}\label{sec:root}
In this section we define quantum analogs of the root vectors $A_n, G_m$ of the Onsager algebra $\ons$ up to scalar multiplication. Quantum analogs of $A_n$ will be called real root vectors, and quantum analogs of $G_m$ will be called imaginary root vectors of $\cb_c$. To mimic Damiani's construction, the quantum analog of $G_m$ will be denoted by $B_{m\delta}$. Similarly, the quantum analog of $A_{n}$ will be denoted $B_{n\delta+\alpha_1}=B_{-(n+1)\delta+\alpha_0}$ for $n\in \Z$, see Section \ref{sec:specialization} for the precise correspondence.
\subsection{Definition of $B_\delta$ and the real root vectors}\label{sec:root-vecs}
Following Damiani's construction \cite[Section 3.1]{a-Damiani93} we set
\beqa
B_\delta = a B_0 B_1 + b B_1 B_0. \nonumber
\eeqa
where $a,b\in \Q(q)$ are parameters which are still to be determined.
For any elements $x,y\in \cb_c$ we write $[x,y]=xy-yx$ to denote their commutator.
One calculates
\beqa
[B_\delta,B_0] &=& q^2 [2]_q c b (T_0(B_1) - B_1) + (a+q^2b) B_0[B_1,B_0] \label{eq:Bdel0}\\
               &=& -q^2 [2]_q c a (T_0^{-1}(B_1) - B_1) - (b+aq^2) [B_0,B_1]B_0, \nonumber
\eeqa
\beqa                 
[B_1,B_\delta] &=& q^2 [2]_q c b (T_1^{-1}(B_0) - B_0) + (a+bq^2) [B_1,B_0]B_1 \nonumber\\ 
               &=& -q^2 [2]_q c a (T_1(B_0) - B_0) - (b+q^2a) B_1[B_0,B_1]. \nonumber
\eeqa
We want to eliminate the terms $[B_i,B_j]B_i$ or $B_i[B_j,B_i]$. Here we have a choice. Indeed, up to an overall factor we can consider either
\beqa
  B_\delta= -B_0 B_1 + q^{-2} B_1 B_0 \qquad \mbox{or} \qquad \tilde{B}_\delta= -B_1 B_0 + q^{-2} B_0 B_1.\label{eq:Bdelta}
\eeqa
Observe that $\Phi(B_\delta)=\tilde{B}_\delta$. As in Damiani's construction the choice of $B_\delta$ dictates the choice of root vectors $B_{n \delta+\alpha_i}$ with $i=0,1$. Indeed, Equation \eqref{eq:Bdel0} suggests to define
\beqa
  B_{\delta+\alpha_0}=T_0(B_1)\qquad \mbox{and} \qquad B_{\delta+\alpha_1}=T_1^{-1}(B_0).\label{eq:Bdelal}
\eeqa
We then obtain
\begin{align}
  [B_\delta,B_0] &= c [2]_q (B_{\delta+\alpha_0} - B_1),\label{eq:BdelB0} \\
  [B_1,B_\delta] &= c [2]_q (B_{\delta+\alpha_1} - B_0).\label{eq:B1Bdel}
\end{align}
We can now check an analog of \cite[Section 3.2, Lemma 1]{a-Damiani93}.
\begin{lem}\label{lem:T1Bdel}
  The following relations hold in $\cb_c$:
  \begin{enumerate}
    \item   $T_1(B_\delta)=\tilde{B}_\delta$ and hence $T_0 \Phi(B_\delta)=B_\delta$.
    \item   For any $n\in\Z$ we have
      \begin{align}
        [B_\delta, (T_0\Phi)^n(B_0)]=  c [2]_q 
        \Big( (T_0\Phi)^{n+1}(B_0) - (T_0\Phi)^{n-1}(B_0)\Big).\label{eq:comm1}
      \end{align}  
  \end{enumerate}  
\end{lem}
\begin{proof}
  To verify the first equation in (1) we calculate 
  \begin{align*}
    T_1(B_\delta)&\stackrel{\phantom{\eqref{eq:qDG}}}{=}-T_1(B_0) B_1 + q^{-2} B_1 T_1(B_0)\\
                 &\stackrel{\phantom{\eqref{eq:qDG}}}{=}\frac{1}{q^2 [2]_q c}\big(-B_0 B_1^3 +q [2]_q B_1 B_0 B_1^2 - q^2 B_1^2 B_0 B_1 \big) - B_0 B_1 \\
       &  \qquad \qquad \qquad        + \frac{1}{q^2 [2]_q c}\big(q^{-2}B_1 B_0 B_1^2 -q^{-1} [2]_q B_1^2 B_0 B_1 + B_1^3 B_0 \big) + q^{-2} B_1 B_0\\ 
       &\stackrel{\phantom{\eqref{eq:qDG}}}{=}\frac{1}{q^2 [2]_q c}\big(B_1^3 B_0 - [3]_qB_1^2 B_0 B_1 + [3]_qB_1 B_0B_1^2 - B_0 B_1^3\big) + q^{-2} B_1 B_0 - B_0 B_1\\
       &\stackrel{\eqref{eq:qDG}}{=}-\frac{[2]_q}{q }  (B_1 B_0-B_0B_1) + (q^{-2} B_1 B_0 - B_0 B_1)\\
       &\stackrel{\phantom{\eqref{eq:qDG}}}{=} -B_1 B_0 + q^{-2} B_0 B_1=\Phi(B_\delta).
  \end{align*}
  The second equation in (1) then follows from \eqref{eq:T1PhiT2Phi}. To verify Equation \eqref{eq:comm1} we apply $(T_0\Phi)^n$ to Equation \eqref{eq:BdelB0} and use (1) and the relations $B_{\delta+\alpha_0}=T_0\Phi(B_0)$ and $B_1=(T_0\Phi)^{-1}(B_0)$. 
\end{proof}
Following Damiani's approach, for any $n\in \Z$ we now define
\begin{align}
  B_{n\delta+\alpha_0}&=(T_0\Phi)^n(B_0),\qquad  B_{n\delta+\alpha_1}=(T_0 \Phi)^{-n}(B_1). \label{eq:Bndelta+alpha}  
\end{align}
Observe that by definition
  \begin{align}\label{eq:Bnd1nd0}
     B_{n\delta+\alpha_1} = B_{-(n+1)\delta+\alpha_0} \qquad \mbox{for all $n\in \Z$}
  \end{align}
and hence $B_\alpha=B_{-\alpha}$ for all real roots. This means that the real root vector $B_\alpha$ for $\alpha\in \roots^\real$ only depends on the image of $\alpha$ under the quotient map $\roots\rightarrow \Rbar$. We frequently write $B_{\overline{\alpha}}$ for $\overline{\alpha}\in \Rbar$.

Relations \eqref{eq:comm1} of the above lemma can now be rewritten as
\begin{align}
  \qquad [B_\delta,B_{n\delta+\alpha_0}]&=c [2]_q\big( B_{(n+1)\delta+\alpha_0} - B_{(n-1)\delta+\alpha_0}\big).\label{eq:BdeltaBreal1}
\end{align}
Using the identification \eqref{eq:Bnd1nd0} we see that \eqref{eq:BdeltaBreal1} is equivalent to
\begin{align}
  [B_\delta,B_{n\delta+\alpha_1}]&= c [2]_q \big( B_{(n-1)\delta+\alpha_1} - B_{(n+1)\delta+\alpha_1}\big).\label{eq:BdeltaBreal2}
\end{align}
\begin{rem}
  In a similar way, we could have defined $\tilde{B}_{n\delta+\alpha_0}=(T_1\Phi)^{-n}(B_0)$ and $\tilde{B}_{n\delta+\alpha_1}=(T_1\Phi)^n(B_1)$.
\end{rem}
\subsection{Definition of the imaginary root vectors $B_{m\delta}$ for $m\ge 2$}\label{sec:imroot}
In  view of the Onsager relations \eqref{eq:Onsager} we aim to construct commuting elements $B_{m\delta}$ for $m\ge 1$. Moreover, in view of Damiani's construction these elements should be fixed by $T_0\Phi$ and of the form
\begin{align*}
  B_{m\delta} = - B_0 B_{(m-1)\delta+\alpha_1} + q^{-2} B_{(m-1)\delta+ \alpha_1} B_0 + \mathrm{l.o.t}
\end{align*}
where $\mathrm{l.o.t}$ denotes terms of lower degree in the generators  $B_0, B_1$ of $\cb_c$. As we will show, the following definition will give the desired properties
\begin{align}\label{eq:Bndelta}
  B_{m\delta} = - B_0 B_{(m-1)\delta+\alpha_1} + q^{-2}B_{(m-1)\delta+\alpha_1}B_0 + (q^{-2}{-}1) C_m
\end{align}
where
\begin{align}\label{eq:Cn}
  C_m=\sum_{p=0}^{m-2} B_{p\delta+\alpha_1}B_{(m-p-2)\delta+\alpha_1}.
\end{align}
Observe in particular that $C_1=0$ and $C_2=B_1^2$.
\subsection{Specialization of root vectors}\label{sec:specialization}
It is well known that the $q$-Onsager algebra $\cb_c$ specializes to the Onsager algebra for $q\rightarrow 1$, see e.g.~\cite[Theorem 10.8]{a-Kolb14}. For any element $x\in \cb_c$ we write $\overline{x}\in \ons$ to denote its specialization if it exists. See \cite[Section 10]{a-Kolb14} for the precise notion of specialization used. By \eqref{eq:D0D1} and \eqref{eq:UqgEmb} one has
\begin{align*}
  \overline{B_0}= \frac{i}{2}A_{-1}, \qquad \overline{B_1}=-\frac{i}{2}A_0
\end{align*}
which by \eqref{eq:Bdelta} implies that $\overline{B_\delta}=G_1$. Comparing \eqref{eq:Onsager} with \eqref{eq:BdeltaBreal1}, \eqref{eq:BdeltaBreal2} one obtains
\begin{align*}
  \overline{B_{n\delta+\alpha_0}}= (-1)^n \frac{i}{2} A_{-n-1}, \qquad
  \overline{B_{n\delta+\alpha_1}}= (-1)^{n+1} \frac{i}{2} A_{n} 
\end{align*}
and hence 
\begin{align*}
  \overline{B_{m\delta}}=(-1)^{m-1}G_m
\end{align*}
by the definition \eqref{eq:Bndelta} of $B_{m\delta}$. 
\section{The PBW theorem for  $\cb_c$}\label{sec:PBWBc}
We now compare the root vectors defined in the previous section to the root vectors defined by Damiani in \cite{a-Damiani93} for the positive part of $U_q(\slzh)$. Using a filtered-graded argument this will allow us to prove a PBW theorem for $\cb_c$ in terms of the specific root vectors defined in Section \ref{sec:root}.
\subsection{The PBW theorem for $U^+$}\label{sec:PBWU+}
Let $U^+$ denote the subalgebra of $U_q(\slzh)$ generated by $E_0, E_1$. Recall that
    \begin{align}\label{eq:roots+}
      {\roots}_+=\{n\delta+\alpha_0, m\delta, n\delta+\alpha_1\,|\,n\in \N_0, m\in \N\}
    \end{align}
denotes the set of positive roots for $\slzh$. Damiani uses Lusztig's braid group action to define root vectors $E_\beta\in U^+$ for every $\beta\in \roots_+$. Recall that Lusztig's braid group automorphisms  
\begin{align*}
  T_{i,L}: U_q(\slzh) \rightarrow U_q(\slzh) \qquad \mbox{for $i\in \{0,1\}$}
\end{align*}
are given by 
\begin{align*}
  T_{i,L}(E_i)&= -F_i K_i, \qquad T_{i,L}(F_i)= -K_i^{-1} E_i,\\
  T_{i,L}(K_i)&= K_i^{-1},\qquad\quad T_{i,L}(K_j)=K_jK_i^2,\\
  T_{i,L}(E_j)&= \frac{1}{[2]_q}(E_i^2 E_j - q^{-1}[2]_q E_i E_j E_i + q^{-2} E_j E_i^2),\\
  T_{i,L}(F_j)&= \frac{1}{[2]_q}(F_j F_i^2  - q [2]_q F_i F_j F_i + q^{2}  F_i^2 F_j)
\end{align*}
for $j\in \{0,1\}$ with $j\neq i$, see \cite[Part VI]{b-Lusztig94}, \cite[Section 2.2]{a-Damiani93}. The additional subscript in the notation $T_{i,L}$ is used to distinguish the Lusztig action from the braid group action on $\cb_c$ defined in Section \ref{sec:auto}.
\begin{rem}\label{rem:not-mere-restriction}
  For $i\in \{0,1\}$ we have
  \begin{align*}
    T_{i,L}(\iota'(B_i))=T_{i,L}(E_i-c_iF_iK_i)=-F_iK_i+c_iq^{-2}E_i\neq \iota'(B_i)
  \end{align*}
  and hence the automorphism $T_i$ given in Section \ref{sec:auto} is not merely the restriction of $T_{i,L}$ to the subalgebra $\cb_c$.
\end{rem}  
The automorphism $\Phi$ of $\cb_c$ is the restriction of an automorphism $\Phi: U_q(\slzh)\rightarrow U_q(\slzh)$ defined by $\Phi(X_0)=X_1$, $\Phi(X_1)=X_0$ for $X=E,F,K^{\pm 1}$. For real roots $n\delta+\alpha_0, n\delta+\alpha_1$ with $n\in \N_0$ Damiani defines
\begin{align*}
  E_{n\delta+\alpha_0}&= (T_{0,L}\Phi)^n(E_0), \qquad  E_{n\delta+\alpha_1}= (T_{0,L}\Phi)^{-n}(E_1).
\end{align*}
For imaginary roots $m\delta$ with $m\in \N$ Damiani sets
\begin{align}\label{eq:Endelta}
  E_{m\delta} = -E_0 E_{(m-1)\delta+\alpha_1} + q^{-2} E_{(m-1)\delta+\alpha_1} E_0.
\end{align}
Consider the ordering $<_D$ on $\roots_+$ given by
\begin{align*}
  n\delta+\alpha_0 &<_D (m+1)\delta <_D l\delta+\alpha_1,& n\delta+\alpha_0 <_D (n+1)\delta+\alpha_0,\\
  (n+2)\delta& <_D (n+1)\delta, & (n+1)\delta+\alpha_1 <_D n\delta+\alpha_1
\end{align*}
for all $l,m,n\in \N_0$.
\begin{thm}\cite[Section 5, Theorem 2]{a-Damiani93}\label{thm:DamPBW}
  The set of monomials 
    \begin{align*}
      \mathscr{B}^+=\{E_{\gamma_1}^{s_1}\cdot \dots \cdot E_{\gamma_M}^{s_M}\,| \,M\in \N_0,\, s_1,\dots,s_M\in \N,\, \gamma_1<_D\dots <_D \gamma_M\in \roots_+\}
    \end{align*}
  is a PBW-basis of $U^+$.  
\end{thm}
The above theorem remains valid for an arbitrary total ordering on $\roots_+$. Indeed, Damiani shows in \cite[Section 4]{a-Damiani93} that for any $\alpha,\beta\in \roots_+$ with $\beta<_D \alpha$ there exist $a_{\alpha \beta}\in \Q(q)\setminus \{0\}$ and $a_{\alpha \beta}^\gamma, b_{\alpha \beta}^{\gamma \delta}\in \Q(q)$ such that
\begin{align}\label{eq:Damiani-LS}
  E_\alpha E_\beta - a_{\alpha \beta} E_\beta E_\alpha = \sum_{\beta<_D\gamma<_D\alpha} a_{\alpha \beta}^\gamma E_\gamma + \sum_{\beta<_D\gamma\le_D\delta<_D\alpha} b_{\alpha \beta}^{\gamma \delta} E_\gamma E_\delta.
\end{align}
In view of the above commutation relation, Theorem \ref{thm:DamPBW} has the following consequence.
\begin{cor}\label{cor:DamPBW}
   Let $<$ be any total ordering on $\roots_+$. The set of ordered monomials 
    \begin{align*}
      \mathscr{B}^+=\{E_{\gamma_1}^{s_1}\cdot \dots \cdot E_{\gamma_M}^{s_M}\,| \,M\in \N_0,\, s_1,\dots,s_M\in \N,\, \gamma_1<\dots < \gamma_M\in \roots_+\}
    \end{align*}
  is a PBW-basis of $U^+$.  
\end{cor}  
\subsection{A natural filtration of $\cb_c$}
  Define an $\N_0$-filtration $\cf^\ast$ on $\cb_c$ such that $\cf^n\cb_c$ consists of the linear span of all monomials of degree at most $n$ in the generators $B_0, B_1$ of $\cb_c$, in other words
  \begin{align*}
    \cf^n\cb_c=\mathrm{Lin}_{\Q(q)}\{B_{i_1}\dots B_{i_k}\,|\, k\le n, \, i_j\in \{0,1\}\, \mbox{for}\, j=1,\dots,k\}.
  \end{align*} 
Let $\mathrm{Gr}_\cf(\cb_c)$ denote the associated graded algebra. Similarly, there exists a filtration $\cg^\ast$ of $U^+$ such that $\cg^n U^+$ consists of the linear span of all monomials of degree at most $n$ in the generators $E_0, E_1$ of $U^+$. It follows from \cite[Proposition 6.2]{a-Kolb14} that
\begin{align}\label{eq:dim-estimate}
  \dim_{\Q(q)}\cf^n \cb_c \ge \dim_{\Q(q)}\cg^n U^+.
\end{align}
The defining relations \eqref{eq:qDG} of $\cb_c$ imply that the associated graded algebra $\mathrm{Gr}_\cf(\cb_c)$ is a quotient of the graded algebra $U^+$. This fact, together with relation \eqref{eq:dim-estimate} implies that
\begin{align}\label{eq:GrU+}
  \mathrm{Gr}_{\cf}(\cb_c)\cong U^+
\end{align}
which means in particular that equality holds in \eqref{eq:dim-estimate}.
For any $n\in \N_0$ let 
  \begin{align*}
    \pi_n:\cf^n \cb_c \rightarrow \cf^n\cb_c / \cf^{n-1}\cb_c
  \end{align*}
denote the natural projection map. Using \eqref{eq:GrU+} we consider the image of the map $\pi_n$ as a subset of $U^+$.  
\subsection{Comparison of root vectors for $\cb_c$ with root vectors for $U^+$}
For any $\gamma=n_0\alpha_0 + n_1 \alpha_1\in \roots_+$ let $n_\gamma=n_0+n_1$ denote the height of $\gamma$.
\begin{prop} \label{prop:grad-fil}
  Let $\gamma\in \roots_+$. The root vector $B_\gamma\in \cb_c$ has the following properties. 
  \begin{enumerate}
    \item\label{it1} $B_\gamma\in \cf^{n_\gamma}\cb_c \setminus \cf^{n_\gamma-1}\cb_c$. 
    \item\label{it2} $\pi_{n_\gamma}(B_\gamma)=\begin{cases}
                             c^{-n} E_\gamma & \mbox{if $\gamma=n\delta+\alpha_0$ or $\gamma=n\delta+\alpha_1$ for some $n\in \N_0$,}\\
                             c^{-m+1} E_\gamma & \mbox{if $\gamma=m \delta$ for some $m\in \N$.}
                                    \end{cases}$
  \end{enumerate}
\end{prop}
\begin{proof}
  Properties (\ref{it1}) and (\ref{it2}) hold for $B_0, B_1$, and $B_\delta$. One now performs induction on $n_\gamma$. Assume that (\ref{it1}) and (\ref{it2}) hold for all $\beta\in \roots_+$ with $n_\beta<n_\gamma$. We first consider the case $\gamma=n\delta+\alpha_0$. In this case, by induction hypothesis, one has $B_{(n-2)\delta+\alpha_0}, B_{(n-1)\delta+\alpha_0}\in \cf^{n_\gamma-2}\cb_c$ and hence \eqref{eq:BdeltaBreal1} implies that
  \begin{align*}
    B_{n\delta+\alpha_0}=\frac{1}{c[2]_q}[B_\delta,B_{(n-1)\delta+\alpha_0}]+ B_{(n-2)\delta+\alpha_0}\in \cf^{n_\gamma}\cb_c. 
  \end{align*}
Moreover, comparing \eqref{eq:BdeltaBreal1} with the relation
  \begin{align*}
    [E_\delta,E_{n\delta+\alpha_0}]=[2]_q E_{(n+1)\delta+\alpha_0}
  \end{align*}
given in \cite[p.~299]{a-Damiani93} one obtains
\begin{align*}  
  \pi_{n_\gamma}(B_{n\delta+\alpha_0})&= \frac{1}{c[2]_q}\big[ \pi_{2}(B_\delta),\pi_{n_\gamma-2}(B_{(n-1)\delta+\alpha_0})\big]\\
  &= \frac{1}{c[2]_q}[E_\delta, c^{-(n-1)}E_{(n-1)\delta+\alpha_0}]\\
  &= c^{-n}E_{n\delta+\alpha_0}.
\end{align*}  
In particular, $B_{n\delta+\alpha_0}\notin \cf^{n_\gamma-1}\cb_c$. This completes the proof of Properties (\ref{it1}) and (\ref{it2}) for $\gamma=n\delta+\alpha_0$. In the case $\gamma=n\delta+\alpha_1$ the statement is proved analogously and hence it holds for all real roots. 

If $\gamma=m\delta$ is an imaginary root then Properties (\ref{it1}) and (\ref{it2}) for real roots and \eqref{eq:Bndelta} imply that $B_{m\delta}\in \cf^{n_\gamma}\cb_c$. Moreover, comparing \eqref{eq:Bndelta} with \eqref{eq:Endelta} one obtains
\begin{align*}
  \pi_{n_\gamma}(B_{m\delta})&= -\pi_{1}(B_0)\pi_{n_\gamma-1}(B_{(m-1)\delta+\alpha_0}) + q^{-2}
  \pi_{n_\gamma-1}(B_{(m-1)\delta+\alpha_0})\pi_{1}(B_0)\\
  &= c^{-m+1}E_{m\delta}.
\end{align*}  
This completes the proof of the proposition.
\end{proof}
As a consequence of Proposition \ref{prop:grad-fil} we immediately obtain the desired PBW theorem for $\cb_c$. 
\begin{thm}\label{thm:PBW}
  Let $<$ denote any total ordering on $\roots_+$. The set of ordered monomials
    \begin{align*}
      \mathscr{B}=\{B_{\gamma_1}^{s_1}\cdot \dots \cdot B_{\gamma_M}^{s_M}\,| \,M\in \N_0,\, s_1,\dots,s_M\in \N_0,\, \gamma_1<\dots < \gamma_M\in \roots_+\}
    \end{align*}
  is a basis of $\cb_c$. 
\end{thm}
\begin{proof}
  For $n\in \N$ consider the set of monomials
   \begin{align*}
      \mathscr{B}_n=\{B_{\gamma_1}^{s_1}\cdot \dots \cdot B_{\gamma_M}^{s_M}\in \mathscr{B}\,| \, \sum_{j=1}^M s_j n_{\gamma_j}\le n\}.
   \end{align*}
   Proposition \ref{prop:grad-fil}(1) implies that $\mathscr{B}_n\subseteq \cf^n\cb_c$. By Proposition \ref{prop:grad-fil}(2) and the PBW Theorem given in Corollary \ref{cor:DamPBW} for $U^+$ the elements of $\mathscr{B}_n$ are linearly independent. Moreover, again by Corollary \ref{cor:DamPBW}, the set $\mathscr{B}_n$ contains $\dim_{\Q(q)}(\cg^n U^+)=\dim_{\Q(q)}(\cf^n \cb_c)$ many elements. Hence $ \mathscr{B}_n$ is a basis of $\cf^n \cb_c$ and the theorem follows. 
\end{proof}
Note that $\roots_+$ is a representative set for the quotient $\Rbar$ defined by \eqref{eq:Rbar}. Hence Theorem \ref{thm:PBW} implies Theorem I as stated in the introduction.
\section{Commutation relations}\label{sec:commutators}
We now turn to the study of commutation relations inside $\cb_c$. This provides $q$-analogs of the classical Onsager relations \eqref{eq:Onsager}. Again one can mimic Damiani's approach to establish commutation relations for the root vectors $B_{\ogamma}$ for $\ogamma\in \Rbar$.
\subsection{A choice of ordering of $\Rbar$}\label{sec:choiceR+}
In Damiani's setting it was beneficial to work with the total ordering $<_D$ of $\roots_+$ because it has the convexity property \eqref{eq:Damiani-LS}. Recall that $\roots_+$ is a representative set for the quotient $\Rbar$ and hence we can identify $\roots_+$ and $\Rbar$. In our setting it is preferable to work with an ordering $<$ of $\Rbar$ for which the basis $\mathscr{B}$ in Theorem \ref{thm:PBW} is invariant under the application of $T_0\Phi$.
This was pointed out to us by one of the referees and simplifies the subsequent discussion significantly. Recall that $(\Z\delta+\alpha_1)\cup \N\delta$ is another representative set of $\Rbar$. We will show in Corollaries \ref{cor:Bdel-commute} and \ref{cor:Bdel-inv} that the imaginary root vectors $B_{m\delta}$ commute pairwise and are invariant under $T_0\Phi$. By \eqref{eq:Bndelta+alpha} we have $(T_0\Phi)(B_{n\delta+\alpha_1})=B_{(n-1)\delta+\alpha_1}$ for all $n\in \Z$. Hence total orderings of $\Rbar\cong(\Z\delta+\alpha_1)\cup \N\delta$ with the desired $T_0\Phi$-invariance are determined by two choices:
\begin{enumerate}
  \item Either the real roots are less than the imaginary roots, or vice versa.
  \item Either $(n+1)\delta+\alpha_1< n\delta+\alpha_1$ for all $n\in \Z$, or vice versa.
\end{enumerate}
The following subtle remark was provided to us by the referee. It shows that the choices in (1) and (2) above are essentially equivalent.
\begin{rem}\label{rem:referee}
  1) It follows from the defining relations \eqref{eq:qDG} of $\cb_c$ that there exists an involutive $\Q(q)$-algebra antiautomorphism $\Psi:\cb_c\rightarrow \cb_c$ fixing $B_0$ and $B_1$. The antiautomorphism $\Psi$ commutes with the automorphism $\Phi$ and has the property that $T_0\Psi=\Psi T_0^{-1}$. Hence
  \begin{align*}
\Phi\Psi(B_{n\delta+\alpha_1})=\Phi\Psi(T_0\Phi)^{-n}(B_1)=(T_0\Phi)^n\Phi(B_1)= B_{n\delta+\alpha_0}=B_{-(n+1)\delta+\alpha_1}
  \end{align*}
  and using the definition \eqref{eq:Bndelta} of $B_{n\delta}$ and the fact that $\Phi\Psi$ is an $\Q(q)$-algebra antiautomorphism one obtains  
  \begin{align*}
    \Phi\Psi(B_{n\delta})=(T_0 \Phi)^{n-1} (B_{n\delta}).
  \end{align*}
  Hence, once we have established that $B_{n\delta}$ is invariant under $T_0\Phi$ in Corollary \ref{cor:Bdel-inv}, we see that $\Phi\Psi$ interchanges the choice in (1) and preserves the choice in (2).

  2)  Assume that the rational function $c=c(q)\in \Q(q)$ satisfies the relation $c(q^{-1})=q^{2a}c(q)$ for some $a\in \Z$. Then there exists an invertible $\Q$-algebra antiautomorphism $\Lambda:\cb_c\rightarrow \cb_c$ such that $\Lambda(q)=q^{-1}$ and $\Lambda(B_i)=q^{a-1} B_i$ for $i=0,1$. The map $\Lambda$ commutes with $\Phi$ and $T_0$. Hence \eqref{eq:Bndelta+alpha} implies that
  \begin{align*}
     \Lambda(B_{n\delta+\alpha_1}) = q^{a-1} B_{n\delta+\alpha_1} \qquad \mbox{for all $n\in \Z$}
  \end{align*}
and \eqref{eq:Bndelta} implies that
  \begin{align*}
     \Lambda(B_{n\delta})=-q^{2a} B_{n\delta} \qquad \mbox{for all $n\in \N$}.
  \end{align*}
The above relations show that the algebra antiautomorphism $\Lambda$ interchanges the choices in both (1) and (2).  
\end{rem}  
In view of the above remark we now specify the choice in (1) and (2). For the remainder of the paper $<$ will denote the total ordering on $\Rbar\cong (\Z\delta+\alpha_1)\cup \N\delta$ uniquely determined by
\begin{align}\label{eq:order-choice}
  (n+1)\delta+\alpha_1 < n\delta+\alpha_1 < m\delta < (m+1)\delta
\end{align}
for all $n\in \Z$, $m\in \N$.

\subsection{Imaginary root vectors and braid group action}
Recall the definition of $B_{n\delta}$ and $C_n$ from Section \ref{sec:imroot}. 
For $n\ge 2$ one has
  \begin{align*}
    (T_0 \Phi)^{-1}(C_{n-1})= \sum_{m=1}^{n-2}B_{m\delta+\alpha_1} B_{(n-m-1)\delta+\alpha_1}
  \end{align*}
and hence
  \begin{align}\label{eq:Cn+1-T0PhiCn-1}
    C_{n+1} - (T_0 \Phi)^{-1}(C_{n-1})&= B_1 B_{(n-1)\delta+\alpha_1} + B_{(n-1)\delta+\alpha_1} B_1.
  \end{align}
Equation \eqref{eq:Cn+1-T0PhiCn-1} will play a crucial role in the proof of the following lemma.   

\begin{lem}\label{lem:ndeltadelta}
  Let $n\in \N$ and assume that
  \begin{align}\label{ass:nI}
  \begin{cases}
        T_0 \Phi (B_{k\delta})=B_{k\delta} & \mbox{for all $k\le n$,}\\
     [B_\delta,B_{k\delta}]=0 & \mbox{for all $k< n$.}\tag{A${}_n$I}
  \end{cases}   
  \end{align}
  Then   $[B_{n\delta}, B_\delta] =  c [2]_q (\id - T_0 \Phi)(B_{(n+1)\delta})$.
\end{lem}
\begin{proof}
  We use Equations \eqref{eq:BdeltaBreal2} and \eqref{eq:Cn+1-T0PhiCn-1} to calculate
  \begin{align}
    [C_n,B_\delta]&= \sum_{m=0}^{n-2}\big( B_{m\delta+\alpha_1}[B_{(n-m-2)\delta+\alpha_1},B_\delta] + [B_{m\delta+\alpha_1},B_\delta] B_{(n-m-2)\delta+\alpha_1}\big)\nonumber\\
      &= c [2]_q \sum_{m=0}^{n-2}\Big(B_{m\delta+\alpha_1}\big(B_{(n-m-1)\delta+\alpha_1} - B_{(n-m-3)\delta+\alpha_1} \big) \nonumber\\
      & \qquad \qquad \qquad\quad+ \big(B_{(m+1)\delta+\alpha_1} - B_{(m-1)\delta+\alpha_1} \big)B_{(n-m-2)\delta+\alpha_1}\Big)\nonumber\\
      &= c [2]_q \Big(B_1 B_{(n-1)\delta+\alpha_1} + B_{(n-1)\delta+\alpha_1} B_1 - B_{(n-2)\delta+\alpha_1} B_0 - B_0 B_{(n-2)\delta+\alpha_1}\nonumber\\
      &\qquad \qquad\quad + 2\sum_{m=1}^{n-2} \big(B_{m\delta+\alpha_1} B_{(n-m-1)\delta+\alpha_1} - B_{(m-1)\delta+\alpha_1} B_{(n-m-2)\delta+\alpha_1}\big)\Big)\nonumber\\
      &= c [2]_q (\id - T_0 \Phi)\Big(B_1 B_{(n-1)\delta+\alpha_1} + B_{(n-1)\delta+\alpha_1} B_1 \nonumber\\
      &\qquad \qquad \qquad\qquad\qquad +2\sum_{m=1}^{n-2} B_{m\delta+\alpha_1} B_{(n-m-1)\delta+\alpha_1} \Big)\nonumber\\
      &= c [2]_q (\id - T_0 \Phi)
      \big(  C_{n+1} + (T_0 \Phi)^{-1}(C_{n-1}) \big).\label{eq:firstBdelta}
  \end{align}
Similarly one calculates
\begin{align}
  [-B_0 & B_{(n-1)\delta+\alpha_1} +q^{-2}B_{(n-1)\delta+\alpha_1}B_0, B_\delta]\nonumber\\
     =&-B_0 [B_{(n-1)\delta+\alpha_1},B_\delta] - [B_0,B_\delta] B_{(n-1)\delta+\alpha_1} 
        +q^{-2}B_{(n-1)\delta+\alpha_1}[B_0, B_\delta]\nonumber \\
     &+ q^{-2}[B_{(n-1)\delta+\alpha_1},B_\delta]B_0\nonumber\\
     =& c [2]_q\Big(-B_0 (B_{n\delta+\alpha_1}-B_{(n-2)\delta+\alpha_1}) - (B_1-B_{\delta+\alpha_0})B_{(n-1)\delta+\alpha_1} \nonumber\\
     &\,\,\,\,\qquad+ q^{-2}B_{(n-1)\delta+\alpha_1} (B_1-B_{\delta+\alpha_0}) + q^{-2} (B_{n\delta+\alpha_1}-B_{(n-2)\delta+\alpha_1})B_0\Big)\nonumber\\
     =& c [2]_q (\id - T_0 \Phi)\Big(-B_0 B_{n\delta+\alpha_1} + q^{-2}B_{n\delta+\alpha_1} B_0 - B_1 B_{(n-1)\delta+\alpha_1} \label{eq:secondBdelta}   \\
     &\qquad \qquad\qquad \qquad \qquad \qquad \qquad \qquad \qquad \qquad
     + q^{-2} B_{(n-1)\delta+\alpha_1} B_1\Big). \nonumber
  \end{align}
Using $(T_0 \Phi)^{-1}(B_{(n-1)\delta})=B_{(n-1)\delta}$, which holds by \eqref{ass:nI}, one obtains
\begin{align*}
   B_{(n-1)\delta}= - B_1 B_{(n-1)\delta+\alpha_1} + q^{-2} B_{(n-1)\delta+\alpha_1}B_1 + (q^{-2}-1)(T_0 \Phi)^{-1}(C_{n-1}).
\end{align*}
Inserting the above relation into \eqref{eq:secondBdelta} and using again $(\id-T_0\Phi)B_{(n-1)\delta}=0$ one gets
\begin{align}
  [-B_0 & B_{(n-1)\delta+\alpha_1} +q^{-2}B_{(n-1)\delta+\alpha_1}B_0, B_\delta]\nonumber\\
    &=c [2]_q (\id - T_0 \Phi)\Big({-}B_0 B_{n\delta+\alpha_1} {+} q^{-2}B_{n\delta+\alpha_1} B_0 {-}(q^{-2}{-}1)(T_0 \Phi)^{-1}(C_{n-1})\Big).\label{eq:thirdBdelta}
\end{align} 
 Finally, we add up Equations \eqref{eq:firstBdelta} and \eqref{eq:thirdBdelta} and obtain
 \begin{align*}
   [-B_0&B_{(n-1)\delta+\alpha_1} + q^{-2}B_{(n-1)\delta+\alpha_1}B_0 +(q^{-2}{-}1)C_n\,,\,B_\delta] \\
    &=c [2]_q (\id - T_0 \Phi)\big( {-}B_0B_{n\delta+\alpha_1} + q^{-2}B_{n\delta+\alpha_1}B_0 +(q^{-2}{-}1)C_{n+1}  \big)
 \end{align*} 
which completes the proof of the lemma. 
\end{proof}
Lemma \ref{lem:ndeltadelta} shows in particular that if (A${}_k$I) holds for some $k\in \N$ and additionally $[B_{\delta},B_{k\delta}]=0$ then (A${}_{k+1}$I) also holds. This observation has the following consequence.
\begin{cor}\label{cor:ndeltadelta}
  Let $n\in \N$ and assume that $[B_\delta,B_{k\delta}]=0$ for all $k<n$. Then $T_0\Phi(B_{k\delta})=B_{k\delta}$ for $k=1,\dots, n$ and
  \begin{align*}
   [B_{n\delta}, B_\delta] =  c [2]_q (\id - T_0 \Phi)(B_{(n+1)\delta}).
 \end{align*}
\end{cor}
Lemma \ref{lem:ndeltadelta} also allows us to rewrite the term $C_m$ given in \eqref{eq:Cn} in ordered form with respect to the ordering $<$ of $\Rbar$ defined in Section \ref{sec:choiceR+}.
For any real number $x\in \R$ we write $[x]$ to denote the largest integer less than or equal to $x$.
\begin{prop}\label{prop:Cm-ordered}
  Let $n\in \N$ and assume that $[B_\delta,B_{k\delta}]=0$ for all $k<n$. Then for all $m\in \N$ with $m \le n+2$ the relation
  \begin{align}\label{eq:Cm-ordered}
    C_m = -\sum_{p=1}^{[\frac{m-1}{2}]}q^{-2(p-1)}B_{(m-2p)\delta} + \sum_{p=1}^{[\frac{m}{2}]}a^m_p B_{(m-p-1)\delta+\alpha_1} B_{(p-1)\delta+\alpha_1}
  \end{align}
  holds, where the coefficients $a_p^m$ are given by
  \begin{align}\label{eq:apm}
    a_p^m=\begin{cases}
       q^{-2(p-1)}(1+q^{-2}) & \mbox{if $p=1,2 \dots,[\frac{m-1}{2}]$,}\\
       q^{-m+2}              & \mbox{if $m$ is even and $p=\frac{m}{2}$.}
          \end{cases} 
  \end{align}
\end{prop}
\begin{proof}
As observed in Subsection \ref{sec:imroot}, Equation \eqref{eq:Cm-ordered} holds for $m=1, 2$. Assume now that \eqref{eq:Cm-ordered} holds for a given $m\le n$. By Corollary \ref{cor:ndeltadelta} one obtains $(T_0\Phi)^{-1}(B_{m\delta})=B_{m\delta}$. Hence 
\begin{align*}
  B_{m\delta} = -B_1 B_{m\delta+\alpha_1} + q^{-2} B_{m\delta+\alpha_1} B_1 + (q^{-2}-1) (T_0\Phi)^{-1}(C_m).
\end{align*}
On the other hand \eqref{eq:Cn+1-T0PhiCn-1} implies that
\begin{align*}
  C_{m+2}=B_1 B_{m\delta+\alpha_1} + B_{m\delta+\alpha_1} B_1 + (T_0\Phi)^{-1}(C_m).
\end{align*}
Adding the above two relations one obtains 
\begin{align*}
  C_{m+2} = -B_{m\delta} + (1+q^{-2}) B_{m\delta+\alpha_1} B_1 + q^{-2} (T_0 \Phi)^{-1}(C_m).
\end{align*}
Using the induction hypothesis for $C_m$ this becomes 
\begin{align*}
  C_{m+2} =& -B_{m\delta} + (1+q^{-2}) B_{m\delta+\alpha_1} B_1 -  \sum_{p=1}^{[\frac{m-1}{2}]}q^{-2p}B_{(m-2p)\delta}\\
     & \qquad + q^{-2}\sum_{p=1}^{[\frac{m}{2}]}a^m_p B_{(m-p)\delta+\alpha_1} B_{p\delta+\alpha_1}\\
     =&  -\sum_{p=1}^{[\frac{m+1}{2}]}q^{-2(p-1)}B_{(m+2-2p)\delta} + \sum_{p=1}^{[\frac{m+2}{2}]}b^{m+2}_p B_{(m-p+1)\delta+\alpha_1} B_{(p-1)\delta+\alpha_1}
\end{align*}
where $b_1^{m+2}=(1+q^{-2})$ and $b_p^{m+2}=q^{-2}a_{p-1}^m$ for $p= 2,3,\dots,[\frac{m+2}{2}]$. As the coefficients $a^m_p$ are given by \eqref{eq:apm} one obtains that $b^{m+2}_p=a^{m+2}_p$ for $p=1,2,\dots,[\frac{m+2}{2}]$. This concludes the induction.
\end{proof}
\subsection{Commutators of real root vectors}
For $r,s\in \Z$ the product $B_{r\delta+\alpha_1}B_{s\delta+\alpha_1}$ is ordered with respect to  the ordering $<$ of $\Rbar$ if and only if $r\ge s$. We can use the expression \eqref{eq:Bndelta} for $B_{m\delta}$ together with formula \eqref{eq:Cm-ordered} to rewrite $B_{r\delta+\alpha_1} B_{s\delta+\alpha_1}$ for $r<s$ in ordered form. Recall that for any $p\in \Q(q)$ and any $x,y\in \cb_c$ we write
\begin{align*}
  [x,y]_p = xy- p\,yx
\end{align*}
to denote the $p$-commutator of $x$ and $y$.
\begin{prop}\label{prop:comRealReal}
  Let $n\in \N$ and assume that $[B_\delta,B_{k\delta}]=0$ for all $k<n$. For all $m\in \N$ with $m\le n$ and all $r\in \Z$ one has
  \begin{align}
    [B_{r\delta+\alpha_1}, B_{(r+m)\delta+\alpha_1}]_{q^{-2}} = &-B_{m\delta} -  (q^{-2}{-}1) \sum_{p=1}^{[\frac{m-1}{2}]}q^{-2(p-1)}B_{(m-2p)\delta} \label{eq:comRealReal2}\\
    &+ (q^{-2}{-}1) \sum_{p=1}^{[\frac{m}{2}]}a^m_p B_{(m+r-p)\delta+\alpha_1} B_{(r+p)\delta+\alpha_1}\nonumber
  \end{align}
where the coefficients $a^m_p$ are given by \eqref{eq:apm}.
\end{prop}
\begin{proof}
  By Corollary \ref{cor:ndeltadelta} we have $T_0\Phi (B_{m\delta})=B_{m \delta}$. Applying $(T_0\Phi)^{-(r+1)}$ to \eqref{eq:Bndelta} one obtains 
 \begin{align*}
     [B_{r\delta+\alpha_1}, B_{(r+m)\delta+\alpha_1}]_{q^{-2}} = -B_{m\delta} + (q^{-2}-1)(T_0\Phi)^{-(r+1)}(C_m).
  \end{align*}
  Now Equation \eqref{eq:comRealReal2} follows from Proposition \ref{prop:Cm-ordered}.
\end{proof}
\subsection{Commutators involving $B_{n\delta}$}
We now describe the commutators of imaginary root vectors $B_{n\delta}$ with any other root vector $B_{\ogamma}$ for $\ogamma \in \Rbar$. For any $n\in \N$ with $n\ge 2$ define
\begin{align}
  F_n=q^{-2}[C_{n-1},B_{\delta+\alpha_0}] - [T_0\Phi(C_n),B_0] - (q^2-q^{-2})B_{\delta+\alpha_0}C_{n-1}\label{eq:notFn}.
\end{align}
As we will see, the elements $F_n$ play a crucial role in the description of the commutators $[B_{n\delta},B_1]$. The following recursive formula for $F_n$ will be proved in Appendix \ref{app:Fn}.
\begin{lem}\label{lem:Fn}
  Let $n\in \N$ with $n\ge 2$ and assume that $[B_\delta,B_{k\delta}]=0$ for all $k<n$. Then one has
  \begin{align*}
    F_{n+1}= B_1B_{n\delta}+ q^2B_{n\delta}B_1 - B_{(n-1)\delta}B_0 - q^2B_0B_{(n-1)\delta} + (T_0\Phi)^{-1}(F_n).
  \end{align*}  
\end{lem}
The following Lemma provides the essential step to determine the commutators $[B_{n\delta}, B_{r\delta+\alpha_1}]$ for all $n\in \N$, $r\in \Z$.
\begin{lem}\label{lem:B1Bndelta}
  Let $n\in\N$ with $n\ge 3$ and assume that $[B_\delta,B_{k\delta}]=0$ for all $k<n$. Then
  \begin{align}
 [B_{n\delta}, B_1]=& q^{-2}[B_{(n-1)\delta},B_{0}] 
 - [B_{(n-2)\delta}, B_1] + q^{2}[B_{(n-1)\delta},B_{\delta+\alpha_1}]\label{eq:Bn+1prefin}\\
& + (q^2-q^{-2})\big(B_{\delta+\alpha_1}-B_0\big)B_{(n-1)\delta}.\nonumber
  \end{align}
\end{lem}
\begin{proof}
By Corollary \ref{cor:ndeltadelta} we have $T_0\Phi(B_{n\delta})=B_{n\delta}$. Using this and the definition of the element $F_n$ in \eqref{eq:notFn} one calculates
\begin{align}
    [B_{n\delta}, B_0]=& [-B_{\delta+\alpha_0} B_{(n-2)\delta+\alpha_1} + q^{-2}B_{(n-2)\delta+\alpha_1}B_{\delta+\alpha_0} + (q^{-2}{-}1)T_0\Phi(C_n), B_0]\nonumber\\
    =& -B_{\delta+\alpha_0} B_{(n-2)\delta+\alpha_1} B_0 + q^{-2}B_{(n-2)\delta+\alpha_1}B_{\delta+\alpha_0}B_0 \nonumber\\
    &+B_0B_{\delta+\alpha_0} B_{(n-2)\delta+\alpha_1} - q^{-2} B_0 B_{(n-2)\delta+\alpha_1}B_{\delta+\alpha_0}\nonumber\\
    & \qquad+(q^{-2}{-}1)[T_0\Phi(C_n), B_0]\nonumber\\
    =& -q^2 B_{\delta+\alpha_0} B_0 B_{(n-2)\delta+\alpha_1} - q^2 B_{\delta+\alpha_0} B_{(n-1)\delta} + q^2 (q^{-2}{-}1)B_{\delta+\alpha_0} C_{n-1}\nonumber\\
    & +q^{-4}B_{(n-2)\delta+\alpha_1} B_0 B_{\delta+\alpha_0} - q^{-2}  B_{(n-2)\delta+\alpha_1} B_\delta\nonumber\\
    &+q^2 B_{\delta+\alpha_0} B_0 B_{(n-2)\delta+\alpha_1} + q^2 B_\delta B_{(n-2)\delta+\alpha_1}\nonumber\\
    &- q^{-4}  B_{(n-2)\delta+\alpha_1} B_0 B_{\delta+\alpha_0} + q^{-2} B_{(n-1)\delta}B_{\delta+\alpha_0} \nonumber\\
    &- q^{-2} (q^{-2}{-}1) C_{n-1}B_{\delta+\alpha_0} +(q^{-2}{-}1)[T_0\Phi(C_n), B_0] \nonumber\\
    =&    q^{-2} [B_{(n-1)\delta},B_{\delta+\alpha_0}] + (q^{-2}{-}q^2) B_{\delta+\alpha_0}
     B_{(n-1)\delta}\nonumber\\
     & + q^{2}  [B_\delta,B_{(n-2)\delta+\alpha_1}]  - (q^{-2}{-}q^2) B_{(n-2)\delta+\alpha_1} B_\delta\nonumber\\
     & - q^{-2} (q^{-2}{-}1) [C_{n-1},B_{\delta+\alpha_0}] - (q^{-2}-q^2)(q^{-2}-1)B_{\delta+\alpha_0} C_{n-1}\nonumber\\
     & \qquad+(q^{-2}{-}1)[T_0\Phi(C_n), B_0]\nonumber\\
    =& q^{-2} [B_{(n-1)\delta},B_{\delta+\alpha_0}] + q^2  [B_\delta,B_{(n-2)\delta+\alpha_1}]
     \label{eq:B0Bn} \\
& \qquad + (q^2-q^{-2})\big(  B_{(n-2)\delta+\alpha_1} B_\delta -B_{\delta+\alpha_0}
     B_{(n-1)\delta}\big) + (1{-}q^{-2})F_n.\nonumber
\end{align}
Replacing $n$ by $n-1$, we know in particular that
\begin{align}
  (1-q^{-2})F_{n-1}&= [B_{(n-1)\delta}, B_0]- q^{-2}  [B_{(n-2)\delta},B_{\delta+\alpha_0}] - q^2 [B_\delta,B_{(n-3)\delta+\alpha_1}]
  \label{eq:Fn-1}\\
  &\qquad - (q^2-q^{-2})\big( B_{(n-3)\delta+\alpha_1} B_\delta -B_{\delta+\alpha_0}
     B_{(n-2)\delta}\big).\nonumber 
\end{align} 
In view of our assumption $n\ge 3$ we can use Lemma \ref{lem:Fn} to obtain
\begin{align*} (1{-}q^{-2})F_{n}&\stackrel{\phantom{\eqref{eq:Fn-1}}}{=}(1{-}q^{-2})\big(B_1B_{(n{-}1)\delta}+ q^2B_{(n{-}1)\delta}B_1 - B_{(n{-}2)\delta}B_0 - q^2B_0B_{(n{-}2)\delta}\big)\\
  &\qquad\qquad + (1{-}q^{-2})(T_0\Phi)^{-1}(F_{n{-}1})\\
  &\stackrel{\eqref{eq:Fn-1}}{=} (1{-}q^{-2})\big( B_1B_{(n{-}1)\delta}+ q^2B_{(n{-}1)\delta}B_1 - B_{(n{-}2)\delta}B_0 - q^2B_0B_{(n{-}2)\delta}\big)\\
   &\quad + (T_0\Phi)^{-1}  \Big( [B_{(n-1)\delta}, B_0] - 
q^{-2} [B_{(n-2)\delta},B_{\delta+\alpha_0}] - q^2 [B_\delta,B_{(n-3)\delta+\alpha_1}] \nonumber \\
& \qquad \qquad - (q^2-q^{-2})\big( B_{(n-3)\delta+\alpha_1}B_\delta -B_{\delta+\alpha_0} B_{(n-2)\delta}\big)\Big)\nonumber\\
&\stackrel{\phantom{\eqref{eq:Fn-1}}}{=}
(1{-}q^{-2})\left(B_1B_{(n-1)\delta}+ q^2B_{(n-1)\delta}B_1 - B_{(n-2)\delta}B_0 - q^2B_0B_{(n-2)\delta}\right)\nonumber\\
& \qquad \qquad +  [B_{(n-1)\delta}, B_1] - 
q^{-2} [B_{(n-2)\delta},B_0] - q^2[B_\delta,B_{(n-2)\delta+\alpha_1}] \nonumber \\
& \qquad  \qquad- (q^2-q^{-2})\big( B_{(n-2)\delta+\alpha_1} B_\delta -B_0 B_{(n-2)\delta}\big)\nonumber\\
&\stackrel{\phantom{\eqref{eq:Fn-1}}}{=} -[B_{(n-2)\delta}, B_0] +q^2 [B_{(n-1)\delta},B_1] -q^2[B_\delta,B_{(n-2)\delta+\alpha_1}] \nonumber\\
& \qquad \qquad +(q^2-q^{-2})\big(B_1 B_{(n-1)\delta} - B_{(n-2)\delta+\alpha_1} B_\delta).\nonumber
\end{align*}
Replacing the above expression in \eqref{eq:B0Bn} and simplifying, we obtain the recursive formula
\begin{align*}
 [B_{n\delta}, B_0]=& q^{-2}[B_{(n-1)\delta},B_{\delta+\alpha_0}] 
 - [B_{(n-2)\delta}, B_0] + q^2[B_{(n-1)\delta},B_1]\\
& + (q^2-q^{-2})\big(B_1 B_{(n-1)\delta} - B_{\delta+\alpha_0}B_{(n-1)\delta}).
\end{align*}
Now we obtain \eqref{eq:Bn+1prefin} by application of $(T_0\Phi)^{-1}$.
\end{proof}
With the preparation provided by Lemma \ref{lem:B1Bndelta} we are now in a position to determine the commutator between real and imaginary root vectors.
\begin{prop}\label{prop:Bb10Bndelbis}
Let $n\in {\mathbb N}$ and assume that $[B_\delta,B_{k\delta}]=0$ for all $k<n$. Then for all $m\in \N$ with $m\le n$ and all $r\in \Z $ the following relation holds:
\begin{align}
& [B_{m\delta},B_{r\delta+\alpha_1}] = 
 c[2]_q \Bigg(  q^{-2(m-1)} B_{(r-m)\delta +\alpha_1} \label{eq:BndBp02}\\ 
&\quad + (q^2-q^{-2}) \sum_{h=0}^{m-2}      q^{2(m-2-2h) } B_{(r+m-2-2h)\delta+\alpha_1} -q^{2(m-1)} B_{(r+m)\delta +\alpha_1} \Bigg)\nonumber\\
&-(q^2-q^{-2}) \sum_{l=1}^{m-1}   \Bigg(   q^{-2(l-1) } B_{(r-l)\delta+\alpha_1} \nonumber\\
&\quad+(q^2-q^{-2}) \sum_{h=1}^{l-1}      q^{-2(l-2h) } B_{(r+2h-l)\delta+\alpha_1} - q^{2(l-1)}B_{(r+l)\delta+\alpha_1}\Bigg)B_{(m-l)\delta}.\nonumber
\end{align}
\end{prop}
\begin{proof}
  By Corollary \ref{cor:ndeltadelta} we have $T_0\Phi(B_{k\delta})=B_{k\delta}$ for all $k\le n$. For $m=1$ formula \eqref{eq:BndBp02} coincides with \eqref{eq:BdeltaBreal2}.
For $m=2$ we use \eqref{eq:Bndelta} and calculate
\begin{align*}
  [B_{2\delta},B_1] &= [-B_0 B_{\delta+\alpha_1} + q^{-2}B_{\delta+\alpha_1}B_0 +(q^{-2}-1)B_1^2, B_1]\\
  &=-B_0 B_{\delta+\alpha_1} B_1 + q^{-2} B_{\delta+\alpha_1} B_0 B_1 + B_1 B_0 B_{\delta+\alpha_1}  - q^{-2} B_1 B_{\delta+\alpha_1} B_0\\
  &=-q^2 B_0 B_\delta -q^{-2 }B_{\delta+\alpha_1} B_\delta + q^2 B_\delta B_{\delta+\alpha_1} + q^{-2} B_\delta B_0 \\
  &=q^2[B_\delta,B_{\delta+\alpha_1}] + (q^2-q^{-2}) B_{\delta+\alpha_1} B_\delta + q^{-2}[B_\delta,B_0] - (q^2-q^{-2}) B_0 B_\delta.
\end{align*}
By formula \eqref{eq:BdeltaBreal2} we hence obtain
\begin{align*}
  [B_{2\delta},B_1]=&c[2]_q \big(q^{-2} B_{-2\delta+\alpha_1} + (q^2-q^{-2})B_1 - q^2 B_{2\delta+\alpha_1}\big)\\
  &\qquad \qquad\qquad \qquad  - (q^2-q^{-2})(B_{-\delta+\alpha_1} - B_{\delta+\alpha_1}) B_\delta.
\end{align*}
Now we obtain \eqref{eq:BndBp02} for $m=2$ and general $r\in\Z$ by application of $(T_0\Phi)^{-r}$.
Performing induction on $m$ we may hence assume that $n\ge 3$ and that \eqref{eq:BndBp02} holds for all $m<n$. Again, it is enough to determine the commutator $[B_{n\delta}, B_1]$. Using the induction hypothesis and Lemma \ref{lem:B1Bndelta} this commutator is shown to coincide with \eqref{eq:BndBp02} by a tedious but straightforward calculation.
\end{proof}

Finally, we want to show that the commutators $[B_{n\delta},B_{m\delta}]$ vanish. 
This will be achieved by an induction over the set of ordered pairs of natural numbers
\begin{align*}
  \N^2_>=\{(n,m)\in \N\times \N\,|\,n>m\}
\end{align*}
with the lexicographic ordering given by
\begin{align}\label{eq:lex}
    (k,l) <_{lex} (n,m)  \qquad \Longleftrightarrow \qquad  k<n \quad \mbox{or ($k=n$ and $l<m$).}
\end{align}
First, however, we make the following preparatory observation.
\begin{lem}\label{lem:X=0}
  Let $X\in\cb_c$ be a noncommutative polynomial in the generators $B_0$ and $B_1$ without a constant term. If $T_0\Phi(X)=X$ and $[X,B_1]=0$ then $X=0$.  
\end{lem}
\begin{proof}
  If $T_0\Phi(X)=X$ and $[X,B_1]=0$ then $[X,B_0]=(T_0\Phi)([X,B_1])=0$ and hence $X$ lies in the center of $\cb_c$. By \cite[Theorem 8.3]{a-Kolb14} the center of $\qOns$ consists of scalars $\mathbb{Q}(q)1$. By assumption $X$ can be written as a noncommutative polynomial in the generators $B_0, B_1$ without a constant term. Such a polynomial can never be transformed into a scalar using only the $q$-Dolan-Grady relations \eqref{eq:qDG}, unless the polynomial vanishes in $\qOns$.
\end{proof}
The proof of the following proposition was suggested to us by one of the referees, replacing a much longer calculation in a previous version of this paper.
\begin{prop}\label{prop:Bdel-commute}
  Let $n\in \N$ and assume that $[B_\delta,B_{k\delta}]=0$ for all $k<n$. Then for all $m_1, m_2\in \N$ with $1\le m_1, m_2\le n$ the following relations hold:
  \begin{enumerate}
    \item $[B_{m_1\delta},[B_{m_2\delta}, B_1]] = [B_{m_2\delta},[B_{m_1\delta}, B_1]]$.
    \item $[B_{m_1\delta}, B_{m_2\delta}]=0$.
  \end{enumerate}  
\end{prop}
\begin{proof}
  Without loss of generality we may assume that $m_1>m_2$. We perform induction over the set of ordered pairs $\N^2_>$ with the lexicographic ordering given by \eqref{eq:lex}. Assume that $[B_{k\delta},B_{l\delta}]=0$ for all $(k,l)\in \N^2_>$ with $k,l\le n$ and $(k,l)<_{lex}(m_1,m_2)$. For any $m=1,\dots,n$ let $\cb_{c,m}$ denote the $\Q(q)$-subalgebra of $\cb_c$ generated by $\{B_{h\delta}\,|\,1\le h < m\}$. By Proposition \ref{prop:Bb10Bndelbis} for any $m=1,\dots, n $ there exist $X_{k,m}\in \cb_{c,m}$ such that
  \begin{align*}
    [B_{m\delta},B_1] = \sum_{k={-m}}^m (T_0\Phi)^{-k} (B_1) X_{k,m}.
  \end{align*}
  By Corollary \ref{cor:ndeltadelta} the elements $X_{k,m}$ and $B_{m\delta}$ are invariant under $T_0\Phi$ for $m\le n$. Moreover, $B_{m_1\delta}$ and $X_{k,m_2}$ commute by induction hypothesis. Hence
  \begin{align*}
    [B_{m_1\delta},[B_{m_2\delta}, B_1]] &= \sum_{k=-m_2}^{m_2} [B_{m_1\delta} , (T_0\Phi)^{-k}(B_1)X_{k,m_2}]\\
    &= \sum_{k=-m_2}^{m_2} (T_0\Phi)^{-k}\big([B_{m_1\delta},B_1]\big)X_{k,m_2}\\
    &=  \sum_{k=-m_2}^{m_2}  \sum_{h=-m_1}^{m_1} (T_0\Phi)^{-h-k}(B_0) X_{k,m_1} X_{k,m_2}.
  \end{align*}
 The latter expression is symmetric in $m_1$ and $m_2$ because $X_{k,m_1}$ and $X_{k,m_2}$ commute by induction hypothesis. This proves the formula in (1).
   Now the Jacobi identity implies that $[[B_{m_1\delta}, B_{m_2\delta}],B_1]=0$ and hence the relation  $[B_{m_1\delta}, B_{m_2\delta}]=0$ follows from Lemma \ref{lem:X=0}.
\end{proof}  
We now obtain the desired commutation of imaginary root vectors by induction over the ordered set $\N_>^2$ using Proposition \ref{prop:Bdel-commute}.
\begin{cor}\label{cor:Bdel-commute}
  For any $m,n\in \N$ one has $[B_{n\delta},B_{m\delta}]=0$. 
\end{cor}
Corollaries \ref{cor:Bdel-commute} and \ref{cor:ndeltadelta} imply that the imaginary root vectors are fixed by $T_0\Phi$.
\begin{cor}\label{cor:Bdel-inv}
  For any $n\in \N$ one has $T_0\Phi(B_{n\delta})=B_{n\delta}$.
\end{cor}
Moreover, Corollary \ref{cor:Bdel-commute} allows us to lift the assumption
\begin{align*}
  [B_\delta, B_{k\delta}]=0 \qquad \mbox{for all $k<n$}
\end{align*}
from Propositions \ref{prop:Cm-ordered}, \ref{prop:comRealReal}, \ref{prop:Bb10Bndelbis}, \ref{prop:Bdel-commute}, from Corollary \ref{cor:ndeltadelta} and from Lemmas \ref{lem:Fn}, \ref{lem:B1Bndelta}.
\begin{cor}\label{cor:q-Ons-rels}
  The relations \eqref{eq:comRealReal2} and \eqref{eq:BndBp02} hold for all $r\in \Z$, $m\in\N$.
\end{cor}
The relations in Corollaries \ref{cor:Bdel-commute} and \ref{cor:q-Ons-rels} provide the desired $q$-analogs of the Onsager relations \eqref{eq:Onsager} and imply Theorem II in the introduction. The explicit form of the terms $C^{\real}_{r,m}$ and $C^{\ima}_{r,m}$ can be read off from Propositions \ref{prop:comRealReal} and \ref{prop:Bb10Bndelbis}.


\begin{appendix}
\section{Proof of Lemma \ref{lem:Fn}}\label{app:Fn}
For any $n\in \N$ define
\begin{align}\label{eq:defRn}
   R_n= q^2 \sum_{m=0}^{n-3} T_0\Phi(C_{n-m-1})B_{m\delta+\alpha_1}
 -q^{-2} \sum_{m=0}^{n-3} B_{m\delta+\alpha_1} T_0\Phi(C_{n-m-1}).
\end{align}
The element $R_n$ will appear in the proof of Lemma \ref{lem:Fn}. Observe that $R_1=R_2=0$ and that
\begin{align*}
  R_3=q^2 B_0^2B_1 - q^{-2}B_1 B_0^2. 
\end{align*}
Using the relation $B_\delta=q^{-2}B_1 B_0 - B_0 B_1$ one can rewrite the above expression as
\begin{align}\label{eq:R3}
  R_3=-B_\delta B_0 - q^2 B_0 B_\delta.
\end{align}
This formula has a generalization for all $n\in \N$.
\begin{lem}\label{lem:Rn}
For any $n\in \N$ one has
\begin{align}\label{eq:recRn}
\qquad R_n= - \sum_{m=0}^{n-3} \big( B_{(n-m-2)\delta}B_{(m-1)\delta+\alpha_1}
 + q^2 B_{(m-1)\delta+\alpha_1} B_{(n-m-2)\delta}\big).
\end{align}
\end{lem}
\begin{proof}
  By \eqref{eq:R3} and the preceding comment we know already that Equation   \eqref{eq:recRn} holds for $n=1,2,3$. Hence, for the remainder of this proof, assume that $n\ge 4$. For any $p\in \N$ with $2\le p \le n-1$ we introduce the notation
  \begin{align}\label{eq:np} 
    (n,n-p) =  q^2 T_0\Phi(C_{p})B_{(n-p-1)\delta+\alpha_1} -q^{-2}  B_{(n-p-1)\delta+\alpha_1} T_0\Phi(C_{p}). 
  \end{align}
  With this notation we can write
  \begin{align}\label{eq:Rnnnp}
    R_n= \sum_{p=2}^{n-1} (n,n-p).
  \end{align}
  For $p=2$ one has
  \begin{align}
    (n,n-2)&\stackrel{\phantom{\eqref{eq:Bndelta}}}{=} q^2 B_0^2 B_{(n-3)\delta+\alpha_1} - q^{-2}B_{(n-3)\delta+\alpha_1} B_0^2\nonumber\\
           &\stackrel{\eqref{eq:Bndelta}}{=} - q^2 B_0 B_{(n-2)\delta} -B_{(n-2)\delta} B_0  + (1{-}q^2)B_0 C_{n-2} + (q^{-2}{-}1) C_{n-2}B_0.\label{eq:nn-2}
  \end{align}
  Similarly, for $p\ge 3$ relation \eqref{eq:Cn+1-T0PhiCn-1} implies that
  \begin{align*}
      T_0 \Phi(C_{p}) =  B_0 B_{(p-3)\delta+\alpha_1} + B_{(p-3)\delta+\alpha_1} B_0 + C_{p-2}
  \end{align*}
  and hence
  \begin{align}
    (n,n{-}p)=& q^2B_{(p-3)\delta+\alpha_1}B_0B_{(n-p-1)\delta+\alpha_1}\label{eq:npint}
 - q^{-2}B_{(n-p-1)\delta+\alpha_1}B_0B_{(p-3)\delta+\alpha_1}\\
& +q^2\big(B_0B_{(p-3)\delta+\alpha_1} + C_{p-2}\big) B_{(n-p-1)\delta+\alpha_1} \nonumber\\
& -q^{-2}B_{(n-p-1)\delta+\alpha_1}\big(B_{(p-3)\delta+\alpha_1}B_0 + C_{p-2}\big).\nonumber
  \end{align}
  Again by \eqref{eq:Bndelta} we have
  \begin{align*}
    B_{(n-p-1)\delta+\alpha_1}B_0 = q^2B_0B_{(n-p-1)\delta+\alpha_1} +q^2 B_{(n-p)\delta} +(q^2-1)C_{n-p}.\nonumber
  \end{align*}
Using this formula in the first line of (\ref{eq:npint}), one obtains
\begin{align}
&(n,n{-}p)= -B_{(n-p)\delta}B_{(p-3)\delta+\alpha_1} - q^2 B_{(p-3)\delta+\alpha_1}B_{(n-p)\delta}  \label{eq:npfin}\\
& \quad \qquad+ q^2B_0B_{(p-3)\delta+\alpha_1}B_{(n-p-1)\delta+\alpha_1} 
- q^{-2}B_{(n-p-1)\delta+\alpha_1}B_{(p-3)\delta+\alpha_1}B_0\nonumber\\
& \quad \qquad + B_{(p-3)\delta+\alpha_1}B_{(n-p-1)\delta+\alpha_1}B_0 - B_0B_{(n-p-1)\delta+\alpha_1}B_{(p-3)\delta+\alpha_1}\nonumber\\
& \quad \qquad +(1-q^2)B_{(p-3)\delta+\alpha_1}C_{n-p} + (q^{-2}-1)C_{n-p}B_{(p-3)\delta+\alpha_1}\nonumber\\
&\quad \qquad +q^2C_{p-2}B_{(n-p-1)\delta+\alpha_1} -q^{-2}B_{(n-p-1)\delta+\alpha_1}C_{p-2}\nonumber
\end{align}
which holds for $p\ge 3$. Using the above expression one obtains for $3\le p\le (n+1)/2$ the relation
\begin{align}
  (n,n{-}p)&+(n,p{-}2)=(n,n{-}p)+(n,n{-}(n{-}p{+}2))\nonumber\\
  =& -B_{(n-p)\delta}B_{(p-3)\delta+\alpha_1} - q^2 B_{(p-3)\delta+\alpha_1}B_{(n-p)\delta} \label{eq:nnpnp2}\\
  &-B_{(p-2)\delta}B_{(n-p-1)\delta+\alpha_1} - q^2 B_{(n-p-1)\delta+\alpha_1}B_{(p-2)\delta} \nonumber\\
  &+(q^2-1) B_0 \big(B_{(n-p-1)\delta+\alpha_1} B_{(p-3)\delta+\alpha_1} + B_{(p-3)\delta+\alpha_1} B_{(n-p-1)\delta+\alpha_1}\big)\nonumber\\
  &+(1-q^{-2}) \big(B_{(p-3)\delta+\alpha_1} B_{(n-p-1)\delta+\alpha_1} + B_{(n-p-1)\delta+\alpha_1} B_{(p-3)\delta+\alpha_1}\big)B_0\nonumber\\
  &+(1-q^2-q^{-2})\big(B_{(p-3)\delta+\alpha_1} C_{n-p} - C_{n-p} B_{(p-3)\delta+\alpha_1}\big)\nonumber\\
  &+(1-q^2-q^{-2})\big(B_{(n-p-1)\delta+\alpha_1} C_{p-2} - C_{p-2} B_{(n-p-1)\delta+\alpha_1}\big). \nonumber
\end{align}
The terms \eqref{eq:nnpnp2} and \eqref{eq:nn-2} cover all terms of the summation in \eqref{eq:Rnnnp} if $n$ is odd. If $n$ is even then the sum \eqref{eq:Rnnnp} contains the additional summand $(n,n-(\frac{n}{2}+1))$. By \eqref{eq:npfin} this summand is given by
\begin{align}
  (n,n{-}(\frac{n}{2}{+}1))  =& -B_{(\frac{n}{2}-1)\delta} B_{(\frac{n}{2}-2)\delta+\alpha_1}- q^2 B_{(\frac{n}{2}-2)\delta+\alpha_1} B_{(\frac{n}{2}-1)\delta}  \label{eq:nnh+1}\\
  &+(q^2-1) B_0 B^2_{(\frac{n}{2}-2)\delta+\alpha_1} + (1-q^{-2})B^2_{(\frac{n}{2}-2)\delta+\alpha_1} B_0 \nonumber\\
 & +(1-q^2-q^{-2})\big( B_{(\frac{n}{2}-2)\delta+\alpha_1} C_{\frac{n}{2}-1}-  C_{\frac{n}{2}-1} B_{(\frac{n}{2}-2)\delta+\alpha_1}\big)\nonumber
\end{align}
Adding up \eqref{eq:nn-2}, \eqref{eq:nnpnp2} for  $3\le p\le (n+1)/2$, and \eqref{eq:nnh+1} if $n$ is even, we obtain
\begin{align*}
  R_n=& -\sum_{p=2}^{n-1}\big(B_{(n-p)\delta}B_{(p-3)\delta+\alpha_1} + q^2 B_{(p-3)\delta+\alpha_1} B_{(n-p)\delta}\big)\\
&+(1-q^2)B_0C_{n-2} +(q^2-1)B_0 \underbrace{\sum_{p=3}^{n-1}\big(B_{(p-3)\delta+\alpha_1} B_{(n-p-1)\delta+\alpha_1} \big)}_{=C_{n-2}}\\
&+ (q^{-2}-1)C_{n-2}B_0 + (1-q^{-2})\underbrace{\sum_{p=3}^{n-1}\big(B_{(p-3)\delta+\alpha_1} B_{(n-p-1)\delta+\alpha_1}\big)}_{=C_{n-2}}B_0\\
& + (1-q^2-q^{-2})\sum_{p=3}^{n-1}\big( B_{(p-3)\delta+\alpha_1}C_{n-p}- C_{n-p}B_{(p-3)\delta+\alpha_1}\big).
\end{align*}
As indicated, the second and the third line of the above expression vanish. Hence, to prove the lemma, it remains to see that the last line of the above expression also vanishes. This follows from the relation
\begin{align}
  \sum_{m=0}^{n-2}C_{n-m} B_{m\delta+\alpha_1}&=\sum_{m=0}^{n-1}\sum_{k=0}^{n-m-2}B_{k\delta+\alpha_1}B_{(n-m-k-2)\delta+\alpha_1} B_{m\delta+\alpha_1}\label{eq:CBBC}\\
  &=\sum_{k=0}^{n-2}B_{k\delta+\alpha_1} C_{n-k}.\nonumber
\end{align}
Replacing $n$ by $n-3$ and shifting the summation index up by $3$, one obtains indeed that the last line of the above expression for $R_n$ vanishes. This completes the proof of the Lemma.
\end{proof}
For later use we note that application of $T_0\Phi$ to Equation \eqref{eq:CBBC} in the above proof gives
\begin{align*}
   \sum_{m=0}^{n-2}T_0\Phi(C_{n-m}) B_{(m-1)\delta+\alpha_1}-\sum_{m=0}^{n-2}B_{(m-1)\delta+\alpha_1} T_0\Phi(C_{n-m})=0,
\end{align*}
which can be rewritten as
\begin{align*}
  \sum_{m=0}^{n-3}T_0\Phi(C_{n-m-1}) B_{m\delta+\alpha_1}-\sum_{m=0}^{n-3}B_{m\delta+\alpha_1} &T_0\Phi(C_{n-m-1})\\
  =& B_0T_0\Phi(C_n) -  T_0\Phi(C_n) B_0.
\end{align*}
Hence we obtain
\begin{align}
   [T_0\Phi(C_n),B_0]= \sum_{m=0}^{n-3}B_{m\delta+\alpha_1} T_0\Phi(C_{n-m-1}) -\sum_{m=0}^{n-3}T_0\Phi&(C_{n-m-1}) B_{m\delta+\alpha_1}\label{eq:com2} 
  \end{align}
which holds for all $n\in \N$. Recall that by definition
\begin{align}\label{eq:combfin}
 F_n=q^{-2}[C_{n-1},B_{\delta+\alpha_0}] - [T_0\Phi(C_n),B_0] - (q^2-q^{-2})B_{\delta+\alpha_0}C_{n-1}
\end{align}
for any $n\in \N$ with $n\ge 2$. Equation \eqref{eq:com2} describes the second commutator in the above expression. With the help of Lemma \ref{lem:Rn} and Equation \eqref{eq:com2} we now provide an alternative formula for the element $F_n$. This formula is the main ingredient needed to prove the recursive formula in Lemma \ref{lem:Fn}.
\begin{lem}\label{lem:FnRn}
  Let $n\in \N$ with $n\ge 2$ and assume that $T_0\Phi(B_{k\delta})=B_{k\delta}$ for all $k\in \N$ with $k\le n-1$. Then one has
\begin{align}
 F_n =&  \sum_{m=0}^{n-3}\big( B_{m\delta+\alpha_1}B_{(n-1-m)\delta} +q^2B_{(n-1-m)\delta}B_{m\delta+\alpha_1} \big)\label{eq:Fn-new}\\
& - \sum_{m=0}^{n-3} \big( B_{(n-m-2)\delta}B_{(m-1)\delta+\alpha_1}
  + q^2 B_{(m-1)\delta+\alpha_1} B_{(n-m-2)\delta}\big).\nonumber
\end{align}
\end{lem}
\begin{proof}
To prove the formula we expand the first commutator in \eqref{eq:combfin}. We have
\begin{align}\label{eq:com1}
\qquad [C_{n-1},B_{\delta+\alpha_0}]=\sum_{m=0}^{n-3}B_{m\delta+\alpha_1}B_{(n-m-3)\delta+\alpha_1} B_{\delta+\alpha_0} -B_{\delta+\alpha_0}C_{n-1}.
\end{align}
By assumption, acting with $T_0\Phi$ on Equation \eqref{eq:Bndelta} gives
\begin{align}\label{eq:T0PhiBnd}
 B_{(k-2)\delta+\alpha_1} B_{\delta+\alpha_0} &=q^2 B_{k\delta} + q^2B_{\delta+\alpha_0} B_{(k-2)\delta+\alpha_1} + (q^2-1)T_0\Phi(C_k)
\end{align}
for any $k\in \N$ with $k\le n-1$.
Using \eqref{eq:T0PhiBnd} for $k=n-m-1$ we obtain
\begin{align}
B_{m\delta+\alpha_1}B_{(n-m-3)\delta+\alpha_1}& B_{\delta+\alpha_0} =   q^2 B_{m\delta+\alpha_1}B_{(n-m-1)\delta}\label{BBB}\\
+ q^2B_{m\delta+\alpha_1} B_{\delta+\alpha_0} & B_{(n-m-3)\delta+\alpha_1}
 + (q^2-1)B_{m\delta+\alpha_1}T_0\Phi(C_{n-m-1}).\nonumber
\end{align}
Using \eqref{eq:T0PhiBnd} for $k=m+2$ in the second term on the right hand side of \eqref{BBB} and inserting the result into \eqref{eq:com1} we get
\begin{align}
   [C_{n-1},&B_{\delta+\alpha_0}]=q^2 \sum_{m=0}^{n-3}B_{m\delta+\alpha_1}B_{(n-m-1)\delta} + q^4 \sum_{m=0}^{n-3}B_{(n-m-1)\delta} B_{m\delta+\alpha_1}\label{eq:com1f}\\
 +& (q^2{-}1) \big(\sum_{m=0}^{n-3} B_{m\delta+\alpha_1}T_0\Phi(C_{n-m-1})
+ q^2 \sum_{m=0}^{n-3} T_0\Phi(C_{n-m-1})B_{m\delta+\alpha_1}\big)\nonumber\\
  +&(q^4{-}1)B_{\delta+\alpha_0}C_{n-1}.\nonumber
\end{align}
The above formula and Equation \eqref{eq:com2} imply that
\begin{align}
  F_n  =& \sum_{m=0}^{n-3}B_{m\delta+\alpha_1}B_{(n-m-1)\delta} + q^2 \sum_{m=0}^{n-3}B_{(n-m-1)\delta} B_{m\delta+\alpha_1}\label{eq:Fn-nearly}\\
  &-q^{-2} \sum_{m=0}^{n-3} B_{m\delta+\alpha_1}T_0\Phi(C_{n-m-1})
  + q^2\sum_{m=0}^{n-3} T_0\Phi(C_{n-m-1})B_{m\delta+\alpha_1}.\nonumber
\end{align}
In view of the definition \eqref{eq:defRn} of the element $R_n$ we obtain
\begin{align*}
   F_n  =& \sum_{m=0}^{n-3}B_{m\delta+\alpha_1}B_{(n-m-1)\delta} + q^2 \sum_{m=0}^{n-3}B_{(n-m-1)\delta} B_{m\delta+\alpha_1} + R_n.
\end{align*}
Now Lemma \ref{lem:Rn} implies Equation \eqref{eq:Fn-new} which completes the proof.
\end{proof}
We are now in a position to prove Lemma \ref{lem:Fn} as an immediate consequence of Lemma \ref{lem:FnRn}.
\begin{proof}[Proof of Lemma \ref{lem:Fn}]
  In view of the assumption $[B_\delta,B_{k\delta}]=0$ for all $k<n$, Corollary \ref{cor:ndeltadelta} implies that
    $T_0\Phi(B_{k\delta})=B_{k\delta}$ for all $k< n+1$.
  Hence Lemma \ref{lem:FnRn} implies that
  \begin{align*}
     F_{n+1}=&  \sum_{m=0}^{n-2}\big( B_{m\delta+\alpha_1}B_{(n-m)\delta} +q^2B_{(n-m)\delta}B_{m\delta+\alpha_1} \big)\\
& - \sum_{m=0}^{n-2} \big( B_{(n-m-1)\delta}B_{(m-1)\delta+\alpha_1}
     + q^2 B_{(m-1)\delta+\alpha_1} B_{(n-m-1)\delta}\big)\\
     =& \sum_{m=0}^{n-3}\big( (T_0\Phi)^{-1}(B_{m\delta+\alpha_1})B_{(n-m-1)\delta} +q^2B_{(n-m-1)\delta}  (T_0\Phi)^{-1}(B_{m\delta+\alpha_1}) \big)\\
&\qquad     +B_1 B_{n\delta} + q^2 B_{n\delta} B_1 - B_{(n-1)\delta} B_0 - q^2 B_0 B_{(n-1)\delta} \\
& - \sum_{m=0}^{n-3} \big( B_{(n-m-2)\delta}  (T_0\Phi)^{-1} (B_{(m-1)\delta+\alpha_1})
            + q^2  (T_0\Phi)^{-1}(B_{(m-1)\delta+\alpha_1}) B_{(n-m-2)\delta}\big)\\
     =&  (T_0\Phi)^{-1}(F_n)  +B_1 B_{n\delta} + q^2 B_{n\delta} B_1 - B_{(n-1)\delta} B_0 - q^2 B_0 B_{(n-1)\delta}.
  \end{align*}
  This proves Lemma \ref{lem:Fn}.
\end{proof}
\end{appendix}


\end{document}